\documentclass[a4paper]{article}

\usepackage{amsthm}
\usepackage{amsmath}
\usepackage{amsfonts}
\usepackage{amssymb}
\usepackage[all]{xy}
\usepackage{color}
\usepackage{soul}

\usepackage{enumerate}   
 
\usepackage{mathtools}

\newcommand{\rk}{{\rm rk}}

\newcommand{\embedFOL}{{\embed_{\L_{\omega \omega}} }}

\newtheorem{theorem}{Theorem}[section]

\newtheorem{lemma}[theorem]{Lemma}
\newtheorem{corollary}[theorem]{Corollary}
\newtheorem{proposition}[theorem]{Proposition}
\newtheorem{Conjecture}[theorem]{Conjecture}
\newtheorem{Question}[theorem]{Question}
\newtheorem{Fact}[theorem]{Fact}
\newtheorem{claim}[theorem]{Claim}
\theoremstyle{definition}
\newtheorem{definition}[theorem]{Definition}
\newtheorem{Example}[theorem]{Example}

\newtheorem{Remark}[theorem]{Remark}

\renewcommand{\DeclareMathOperator}[1]{\newcommand{#1}}

\newcommand{\Ord}{\mathrm{Ord}}

\DeclareMathOperator{\Ve}{\mathbf{V}}

\newcommand{\B}{{\mathcal{B}}}
\newcommand{\He}{{\mathbf{H}}}
\newcommand{\A}{{\mathcal{A}}}
\newcommand{\M}{{\mathcal{M}}}

\newcommand{\N}{{\mathcal{N}}}

\newcommand{\K}{{\mathcal{K}}}

\newcommand{\SR}{{\mathcal{SR}}}

\renewcommand{\L}{{\mathcal{L}}}
\newcommand{\embed}{\preccurlyeq}

\newcommand{\USR}{{\:\mathcal{USR}}}

\newcommand{\till}{{\upharpoonright}}

\newcommand{\Lww}{{ \L_{\omega\omega} }}
\newcommand{{\LI}}{{\L_{\mathrm{I}}}}
\newcommand{\Cd}{\mathrm{Cd}}
\newcommand{\PwSt}{\mathrm{PwSt}}
\newcommand{\Pow}{\wp}
\newcommand{\Q}{{\mathcal{Q}}}

\newcommand{\Mod}{{\rm Mod}}

\newcommand{\LST}{{{\sf LST}}}

\newcommand{\ULST}{{{\sf ULST}}}

\newcommand{\dep}{{\rm dep}}

\newcommand{\ZFC}{{\rm \sf ZFC}}

\newcommand{\vir}[1]{``#1''}

\newcommand{\p}{\medskip \noindent}

\newcommand{\DDelta}{{\underline{\Delta}}}
\newcommand{\SSigma}{{\underline{\Sigma}}}
\newcommand{\PPi}{{\underline{\Pi}}}

\newcommand{\leqqq}{\leftrightarrow}
\newcommand{\trcl}{{\rm trcl}}

\newcommand{\id}{{\rm id}}

\makeatletter
\let\@fnsymbol\@arabic
\makeatother

\title{Bounded Symbiosis and Upwards Reflection} %The large cardinal strength of L\"owenheim-Skolem theorems \\
\author{Lorenzo Galeotti\thanks{ Amsterdam University College, Postbus 94160, 1090 GD Amsterdam, The Netherlands} \hskip0.5mm \thanks{ Institute for Logic, Language and Computation, Universiteit van Amsterdam, Postbus 94242, 1090 GE Amsterdam, The Netherlands}, 
Yurii Khomskii\footnotemark[1] \hskip0.5mm \thanks{Universit\"at Hamburg, Bundesstra{\ss}e 55, 20146 Hamburg} \hskip0.5mm \thanks{This author has received funding from the European Union’s Horizon 2020 research and innovation programme under the Marie Sk\l odowska-Curie grant agreement No 706219 (REGPROP).}, 
Jouko V\"a\"an\"anen\footnotemark[2] \hskip0.5mm \footnote{Department of Mathematics and Statistics, University of Helsinki, Finland} \hskip0.5mm \thanks{The author would like to thank the Academy of Finland, grant no.: 322795}}

\begin{document}

\maketitle

\begin{abstract}
In \cite{Symbiosis}, Bagaria and V\"a\"an\"anen developed a  framework for studying the large cardinal strength of \emph{downwards} L\"owenheim-Skolem theorems and related set theoretic reflection properties. The main tool was the notion of \emph{symbiosis}, originally introduced by the third author in \cite{JoukoThesis, Jouko1979}.

\emph{Symbiosis} provides a way of relating model theoretic properties of strong logics to definability in set theory. In this paper we continue the systematic investigation of \emph{symbiosis} and apply it to \emph{upwards} L\"owenheim-Skolem theorems and reflection principles. To achieve this, we need to adapt the notion of \emph{symbiosis} to a new form, called \emph{bounded symbiosis}. As one easy application, we obtain  upper and lower bounds for the large cardinal strength of upwards L\"owenheim-Skolem-type principles for second order logic. 
\end{abstract}

\section{Introduction} \label{SectionIntro}

Mathematicians have two ways of characterizing a class $\mathcal{C}$ of mathematical structures: \emph{definining} the class in set theory, or \emph{axiomatizing} the class by sentences in logic. Symbolically: \begin{enumerate}
\item   $\Phi(\A)$, where $\Phi$ is a formula in the language of set theory,  vs.
\item  $\A \models \varphi$, where $\varphi$ is a sentence  in some logic. \end{enumerate}
In general, set theory is much more powerful than first order logic. 

%For example, the class of all infinite structures is easy to define in set theory but cannot be axiomatized by first-order logic. The same applies for the class of well-orders or the real number line. 
However, by  \emph{restricting} the  allowed complexity of $\Phi$ on one hand, while considering \emph{extensions}   of first-order logic on the other hand, one gets a more interesting picture.  \emph{Symbiosis} aims to capture an equivalence in strength between set-theoretic definability and model-theoretic axiomatisability. One application of this is connecting properties of some strong logic $\L^*$ to specific set-theoretic principles (often expressed in terms of large cardinals). \emph{Symbiosis}  was first introduced by the third author in \cite{Jouko1979}, and studied further in  \cite{MagidorJouko,BagariaCn,Symbiosis}.

\bigskip

%The meta-logical properties of first order logic, such as completeness
%and compactness, imply that there are limits to its expressivity: many natural classes of structures cannot be axiomatised by first order formulas. E.g., there is no first order formula $\phi$ such that a structure $\A$ is a well-order if and only if $\A\models\phi$; the same applies for the class of complete orders or the real number line.

%Of course, all of these classes can easily be defined in first order set theory. In particular, there is a $\Delta_1$ formula $\Phi$ such that $\A$ is a well-order if and only if $\Phi(\A)$.

If $\A$ is a structure and $\phi$ a first-order formula, then the statement ``$\A \models \phi$'' is $\Delta_1$ in set theory. Therefore every first-order axiomatizable class of structures, i.e., every class of the form $\Mod(\phi) = \{\A \; : \; \A \models \phi\}$, is $\Delta_1$-definable. 

The converse does not hold: for example, the class of all well-ordered structures is easily seen to be   $\Delta_1$-definable but not first-order axiomatizable. So it is  natural to look for a  logic $\L^*$ extending first order logic, with the property that every $\Delta_1$-definable class would be axiomatizable by an $\L^*$-sentence

%This question was studied by V\"a\"an\"anen in \cite{asdasdasdasdasd} who introduced the notion of \emph{symbiosis} to study the relationship between axiomatisability and definability: informally, a logic $\L$ and a class of formulas $\Delta$ are said to be \emph{symbiotic} if $\L$-axiomatisability coincides with $\Delta$-definability (for a  precise definition, see Definition \ref{def:symbiosis}). The concept of symbiosis was studied, e.g., in \cite{Jouko1979,MagidorJouko,BagariaCn,Symbiosis}.
%\bigskip
%Staying with the example of well-orders, suppose we extend first order logic with a quantifier which can describe cardinalities. 

Consider the logic $\LI = \L_{\omega \omega}(\mathrm{I})$ obtained from first order logic $\L_{\omega\omega}$ by adding the  \emph{H\"artig quantifier} $\mathrm{I}$, defined by $$\mathcal{A} \models \mathrm{I}xy \; \phi(x) \psi(y)\mbox{ iff }|\{a \in A \,: \, \A \models \phi[a]\}| = |\{b \in A \, : \, \A \models \psi[b]\}|$$
and  consider its closure under the so-called  $\DDelta$-operator (Definition \ref{DefDelta}).\footnote{Usually the symbol used here is a simple $\Delta$, but in this paper we choose  the symbol $\DDelta$, and similarly $\SSigma$, in order to easily distinguish the model-theoretic notions from the L\'evy complexity of formulas in the language of set theory, i.e.,  $\Sigma_n$ and $\Delta_n$ formulas.}  We then obtain a logic, which we will call $\DDelta(\LI)$, such that every $\Delta_1$-definable class, if closed under isomorphisms, is $\DDelta(\LI)$-axiomatisable (see Proposition \ref{symbiosis2} or \cite[Example 2.3]{Jouko1979}).

However, $\DDelta(\LI)$-axiomatisability is now too strong to be ``symbiotic'' with $\Delta_1$-definability:  the class $$\{(A,P) \;\:  : \:  \; |\{x \in A \, : \, P(x)\}| = |\{x \in A \, : \, \lnot P(x)\}|\}$$ is not  $\Delta_1$ (it is not absolute), but it is axiomatisable in $\LI$ by the sentence
$$\mathrm{I}xy (P(x)) (\lnot P(y)).$$

One can observe that all $\Delta(\LI)$-axiomatisable classes are $\Delta_2$-definable, but once more, there are $\Delta_2$-definable classes that are not $\Delta(\LI)$-axiomatisable (see Figure \ref{stairs}).

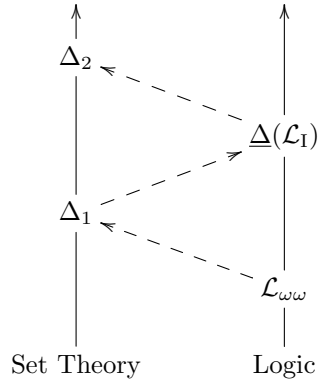
\begin{figure}[h]
$$
\xymatrix@C=0.5cm@R=0.5cm{
%\xymatrix{
 & &   \\
 \Delta_2 \ar@{-}[dd]   \ar@{->}[u] & &    \\
 & & \DDelta(\LI)   \ar@{-}[dd]\ar@{->}[uu]\ar@{-->}[ull] \\
\Delta_1\ar@{-->}[urr]  \ar@{-}[dd]  & &  \\
  & & \L_{\omega \omega} \ar@{-->}[ull]  \ar@{-}[d] \\ 
\text{Set Theory} & & \text{Logic}}
$$
\caption{Set-theoretic definability vs. axiomatization in a logic} \label{stairs} \end{figure}

Interesting \emph{symbiosis} relationships take place for complexity levels $\Delta_1(R)$, for fixed predicates $R$.  In this paper, we will focus on $\Pi_1$ predicates $R$, so the complexity levels will lie below $\Delta_2$. Many such relations have been established in \cite{Jouko1979,Symbiosis}. To name some prominent examples, let $\L^2$ be second order logic with full semantics, and let 
$\L_\mathsf{WF}$ be the logic obtained from $\L_{\omega\omega}$ by adding the generalized quantifier $\mathsf{WF}$ defined by 
\begin{align*}
\A\models %_{\L_\mathsf{WF}}
\mathsf{WF}xy \: \phi(x,y)  \mbox{ iff } 
& \{(x,y)\in A\times A\,:\, \A\models %_{\L_\mathsf{WF}} 
\phi(x,y)\}\mbox{ is well-founded.}
\end{align*}
Furthermore, let 
$\Cd(x)$ be the $\Pi_1$ predicate ``$x$ is a cardinal'',
and let $\PwSt(x,y)$ be the $\Pi_1$ predicate ``$y = \Pow(x)$''. Then we have the \emph{symbiosis} relationships depicted in Figure \ref{figuresymbiosis} (see Propositions \ref{symbiosis1}, \ref{symbiosis2} and \ref{symbiosis3}).

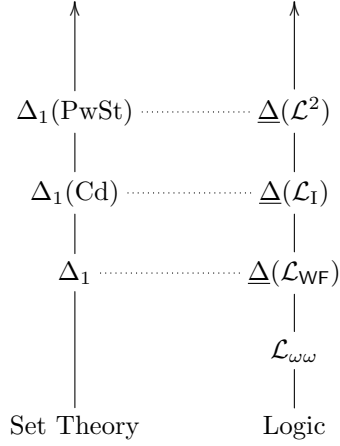
\begin{figure}[h]

$$
\xymatrix@C=0.5cm@R=0.5cm{
%\xymatrix{
  & & \\
 & &   \\
 \Delta_1(\PwSt)  \ar@{-}[d]  \ar@{..}[rr]  \ar@{->}[uu] & & \DDelta(\L^2)  \ar@{->}[uu] \ar@{-}[d] \\
\Delta_1(\Cd) \ar@{-}[d]  \ar@{..}[rr] & & \DDelta(\LI)   \ar@{-}[d] \\
\Delta_1  \ar@{..}[rr] \ar@{-}[dd]  & &  \DDelta(\L_{\mathsf{WF}}) \ar@{-}[d]\\
& & \L_{\omega \omega}  \ar@{-}[d] \\
\text{Set Theory} & & \text{Logic}}
$$
\caption{Symbiosis relations} \label{figuresymbiosis}\end{figure}

As an application of \emph{symbiosis}, Bagaria and V\"a\"an\"anen \cite{Symbiosis} considered the following principles: % (see also    \cite{Jouko1979,MagidorJouko,BagariaCn}): %. allows  us to gauge the precise level of set-theoretic complexity  corresponding to axiomatisability in a given logic, and this can be applied in order to relate meta-logical properties of the logic (such as the L\"owenheim-Skolem theorems)
%to set-theoretic properties of classes of a certain complexity.  For example, the following concepts were defined and studied in \cite{Jouko1979,MagidorJouko,BagariaCn,Symbiosis}:

\begin{definition} 
The \emph{downward L\"owenheim-Skolem-Tarski number}  $\LST(\L^{*})$ is the smallest cardinal $\kappa$ such that for all $\phi \in \L^{*}$, if $\A \models_{\L^{*}} \phi$ then there exists a substructure $\B \subseteq \A$ such that $|B| \leq \kappa$ and $\B \models_{\L^{*}} \phi$. If such a $\kappa$ does not exist, $\LST(\L^{*})$ is undefined.
\end{definition}

\begin{definition} \label{SR} Let $R$ be a predicate in the language of set theory. The \emph{structural reflection number} $\SR(R)$ is the smallest cardinal $\kappa$ such that for every $\Sigma_1(R)$-definable class $\K$ of models in a fixed vocabulary, for every $\A \in \K$ there exists $\B\in\K$ with  $|B| \leq \kappa$ and a first-order elementary embedding $e: \B \; \embed \;   \A$. If such a $\kappa$ does not exists,  $\SR(R)$ is undefined. 
\end{definition}

%{\color{red}Note that both ``principles'' have to changed to ``numbers''. It propose that you do this throughout. Note also that there was a symbol $\R$ which was undefined.}

\begin{theorem}[Bagaria \& V\"a\"an\"anen \cite{Symbiosis}] \label{aaa} Suppose $\L^{*}$ and $R$ are symbiotic. Then $\LST(\L^{*}) =\kappa$ if and only if $\SR(R)=\kappa$. \end{theorem}
\begin{proof}
See \cite[Theorem 6]{Symbiosis}.
\end{proof}

Theorem \ref{aaa} links a meta-logical property of a strong logic to a reflection principle in set theory. Depending on the predicate $R$, the principle $\SR(R)$ has a varying degree of large cardinal strength. In fact, Definition \ref{SR} may be regarded as a kind of Vop\v{e}nka principle, restricted to classes of limited complexity.\footnote{See \cite[Sections 3, 4]{BagariaCn} for more on the connection between $\SR(R)$ and Vop\v{e}nka-type principles.} Indeed, in \cite{Symbiosis} Theorem \ref{aaa} was used to compute  the large cardinal strength of $\SR(R)$ and $\LST(\L^{*})$ for various symbiotic pairs.

%.see also 
%\begin{definition} \emph{Vop\v{e}nka's principle} is the statement ``for every proper class $\K$ of models of a fixed vocabulary, there are distinct $\M,\N \in \K$ with an elementary embedding $e:\M \embed_{\L_{\omega\omega}}  \N$." \end{definition}
%Variants of Vop\v{e}nka's principle have also been considered; e.g., restriction to subsets of $\Ve_\kappa$ or to classes of a specific complexity. See \cite[Section 4]{BagariaCn} for more on the connection between $\SR$ and Vop\v{e}nka principles. The main application of Theorem \ref{aaa} is the computation of  the large cardinal strength of the corresponding statements $\SR(R)$ and $\LST(\L^{*})$.

\bigskip
In this paper we continue the work of Bagaria and V\"a\"an\"anen by developing a framework for the study of  \emph{upward} L\"owenheim-Skolem and  reflection principles. These principles are also interesting because they are closely related to the \emph{compactness}  of the strong logic, although the two notions are not equivalent, and in this paper we do not consider compactness explicitly.  The main innovation of the current work is that, in order to deal with upwards rather than downwards reflection, we need to adapt the notion of \emph{symbiosis}. %, as well as certain other notions. These notions seem interesting in their own right and may have broader applications. %There are several difficulties that must be overcome in order to adapt the \emph{ symbiosis} framework in such a way as to handle such upwards reflexThere are several difficulties in generalizing the xxxxxxxx

\bigskip
The paper is organized as follows: in Section \ref{SectionPreliminaries} we introduce the necessary terminology and some background and in Section   \ref{SecSymbiosis} we present the notion of \emph{symbiosis}.   In Section \ref{SecBoundedSymbiosis} we  introduce  \emph{bounded symbiosis}. Section \ref{SecExamples} is devoted to  examples of \emph{ bounded symbiosis}, and in  Section   \ref{SecCountable} we  prove  the main theorem, showing that under  appropriate conditions the \emph{upward} L\"owenheim-Skolem number corresponds to a suitable upwards set-theoretic reflection principle. Finally, in   Section   \ref{SectionLargeCardinals} we apply our results to compute upper and lower bounds for the {upward} L\"owenheim-Skolem number of second order logic and the corresponding reflection principle, noting that this also provides an upper bound for all other $\Pi_1$ predicates.

\bigskip 

This paper contains research carried out by the first author as part of his PhD Dissertation. Some details have been left out of the paper for the sake of easier readability and navigation. The interested reader may find these details in \cite[Chapter 6]{GaleottiThesis}.

\bigskip
\section{Abstract Logics} \label{SectionPreliminaries}
%\section*{Symbiosis}

We assume that the reader is familiar with standard set theoretic and model theoretic notation and terminology. We will consider abstract logics $\L^*$, without providing a precise definition for what counts as a ``logic''. Typical examples are infinitary logics $\L_{\kappa \lambda}$, full second-order logic $\L^2$, and various extensions of first-order logic by \emph{generalized quantifiers}. For a more detailed analysis we refer the reader to \cite[Chapter 6]{GaleottiThesis} and \cite{BarwiseBook}. Here we only want to stress two important points.

\bigskip

First, we will generally work with  \emph{many-sorted} languages,  using the symbols $s_0, s_1, \dots$ to denote  \emph{sorts}. In this setting, a \emph{domain} may be a collection of  domains (one for each sort), and all constant, relation and function symbols must have their sort specified in advance. This is a matter of convenience, since many-sorted logic can be simulated by standard single-sorted logic by introducing additional predicate symbols. The following definition is essential for what follows. 

%However, for the following definition, it is important to keep in mind that we are dealing with many-sorted logic: %thBut it is important to keep in mind for the following definition: %But for the following definition,  %But we need to be explicit with some definitions which need to take many-sorted languages into account. 

\begin{definition} Suppose that $\tau \subseteq \tau'$ are many-sorted vocabularies and that $\M$ is a $\tau'$-structure. The \emph{reduct} (or \emph{projection})  of $\M$ to $\tau$, denoted by  $ \M \till \tau$, is the structure whose domains are those domains in the sorts available in $\tau$, and the interpretations of all symbols not in $\tau$ are ignored.
\end{definition} 
So a reduct $\M {\till} \tau$ can have a smaller domain, and a smaller cardinality, than the original model $\M$.

\bigskip
Secondly, we should note that one needs to be careful with the syntax of a given logic, because an unrestricted  use of syntax may give rise to some undesirable effects. Consider, for example, an arbitrary set $X \subseteq \omega$, and a vocabulary $\tau$ which has $\omega$-many relation symbols $\{R_i \; : \; i <\omega\}$, such that the arity of $R_i$ is $1$ if $i \in X$ and $2$ if $i \notin X$. The information about arities of relation symbols must be encoded in the vocabulary $\tau$. Therefore, the set $X$ can be computed from $\tau$. 

In infinitary logic, we can encode a set $X \subseteq \omega$ even with finite vocabularies. Let  $\phi_0,\phi_1,\ldots$ be a recursive  enumeration of  $\L_{\omega\omega}$-sentences in some fixed $\tau$, and consider the $\L_{\omega_1\omega}$-sentence  $\Phi := \bigwedge_{n\in X}\phi_n$.  Then $X$ can be computed from $\Phi$.

%For some of our results, it will be important to avoid such undesired side-effects, so we will need to take care about restricting the syntax. % be careful about the syntax in order to avoid problems such as this. %, see {\color{red} refer to later chapter}.

\begin{definition} Let $\L^{*}$ be a logic. The \emph{dependence number} of $\L^{*}$, ${\rm dep}(\L^{*})$, is the least $\lambda$ such that for any vocabulary $\tau$ and any $\L^*$-formula $\phi$ in $\tau$, there is a sub-vocabulary $\sigma\subseteq \tau$ such that $|\sigma| < \lambda$ and $\phi$ only uses symbols in $\sigma$. If such a number does not exist, ${\rm dep}(\L^{*})$ is undefined.  \end{definition}

\begin{definition} We say that a logic $\L^*$ has \emph{$\Delta_0$-definable syntax} if every $\L^*$-formula is $\Delta_0$-definable in set theory (as a syntactic object), possibly with the vocabulary of $\phi$ as parameter. \end{definition}

In our main Theorem \ref{SymbiosisTheorem}, we will restrict attention to logics with a $\Delta_0$-definable syntax and $\dep(\L^*) =  \omega$. Note that this includes all finitary logics obtained by adding finitely many generalized quantifiers to first- or second-order logic.%, and satisfying some definability assumptions on the vocabulary, to avoid problems described above. % This ensures that sentences in these logics are essentially finite objects.

\bigskip

We end this section by defining a version of the upward L\"owenheim-Skolem number for abstract logics.

\begin{definition}[Upward L\"owenheim-Skolem number] \label{upwn} Let $\L^{*}$ be a logic. 
 \begin{enumerate}
\item  The \emph{upward L\"owenheim-Skolem number of $\L^{*}$ for ${<}\lambda$-vocabularies}, denoted by $\ULST_\lambda(\L^{*})$, is the smallest cardinal $\kappa$ such that

\begin{quote} for every vocabulary $\tau$ with $|\tau| < \lambda$ and every $\phi$ in $\L^{*}[\tau]$, if there is a model $\A \models \phi$ with $|\A| \geq \kappa$, then for every $\kappa' > \kappa$, there is a model $\B \models \phi$ such that $|\B|  \geq \kappa'$ and $\A$ is a substructure of $ \B$. \end{quote}

As usual, if there is no such cardinal then  $\ULST_\lambda(\L^{*})$ is undefined.

\item  The \emph{upward L\"owenheim-Skolem number} of $\L^{*}$, denoted by $\ULST_\infty(\L^{*})$ is the smallest cardinal $\kappa$ such that $\ULST_\lambda(\L^{*})\leq \kappa$ for all cardinals $\lambda$. Again, if there is no such cardinal then $\ULST_\infty(\L^{*})$ is undefined.

\end{enumerate}
\end{definition}

Notice that when $\dep(\L^*) = \lambda$, then $ \ULST_\lambda(\L^*) = \kappa$ implies $ \ULST_\infty(\L^*) = \kappa$. In general, $\ULST_\infty(\L^*)$ may fail to be defined even if all $ \ULST_\lambda(\L^*) = \kappa$ are defined. 

Recall also that the \emph{Hanf-number} of a logic is defined analogously to Definition \ref{upwn} but without the assumption that $\A$ is a substructure of $\B$. This additional assumption is rather crucial: it is easy to see that if the dependence number of a logic is defined, then the Hanf number is also defined (see \cite[Theorem 6.4.1]{BarwiseBook}). However, as we shall see in  Section \ref{SectionLargeCardinals}, the existence of upward L\"owenheim-Skolem numbers in the sense of Definition \ref{upwn}, even for logics with dependence number $\omega$, implies the existence of large cardinals. 

%Also, note that, the upward L\"owenheim-Skolem theorem we introduced is stronger than the principle usually take into consideration to compute the Hanf-number of a logic. In the upward L\"owenheim-Skolem theorem we require that $\A$ is elementary in $\B$ while for the Hanf-number there is no required relationship between $\A$ and $\B$. It is easy to see that every logic has an Hanf-number and that this principle is therefore not of large cardinal strength; see, e.g., \cite{BarwiseBook}. As we will see, the additional requirement that we have in the upward L\"owenheim-Skolem theorem makes it into a large cardinal property. 
 
% Usually, the size of the vocabulary is not relevant for these types of principles, for the following reason:
%
%\begin{definition} Let $\L^{*}$ be any logic. The \emph{dependence number} of $\L^{*}$, ${\rm dep}(\L^{*})$, is the least $\lambda$ such that for any vocabulary $\tau$ and any formula $\phi \in \L^{*}[\tau]$, there is another vocabulary $\sigma$ with $|\sigma| < \lambda$ such that (modulo substitution), $\phi \in \L^{*}[\sigma]$. \end{definition} 
%
%
%\begin{lemma} Let $\L^{*}$ be a  logic with ${\rm dep}(\L^{*}) = \lambda$. Then $\ULST_\lambda(\L^{*}) = \kappa \; \to \; \ULST_\infty(\L^{*}) = \kappa $.\end{lemma}

\section{Symbiosis} \label{SecSymbiosis}

\emph{Symbiosis} was introduced by the third author in \cite{Jouko1979}. To motivate its definition, let  $\L^*$  be a logic and $R$ a predicate in set theory. The aim is to establish an equality in strength between $\L^*$-axiomatizability and $\Delta_1(R)$-definability. One direction should be the statement ``the satisfaction relation $\models_{\L^*}$ is $\Delta_1(R)$-definable'', or, equivalently, ``every $\L^*$-axiomatizable class of structures is $\Delta_1(R)$-definable.''

 The converse direction should say, roughly speaking, that every $\Delta_1(R)$-definable class is $\L^*$-axiomatizable. This cannot literally work, because $\L^*$-model classes are closed under isomorphisms whereas this is not necessarily true for arbitrary $\Delta_1(R)$ classes. Therefore we try the approach ``every $\Delta_1(R)$-definable class closed under isomorphisms is $\L^*$-axiomatizable.'' 
 
Unfortunately, this does not always work: \emph{symbiosis } can { only } be established for  logics that are closed under the $\DDelta$-operation. This operation closes the logic under operations which are in a sense ``simple'' but not as simple as mere conjunction, negation, and whatever operations the logic has. In order to define the property of a structure $\A$ being in 
 a model class based merely on the  knowledge that the class is $\Delta_1(R)$, the only way seems to be to use the means of the logic
to build a piece of the set theoretic universe around $\A$, and work in the small universe.  The $\DDelta$-operation is then used to eliminate the extra symbols used to build the small universe.   See \cite{Symbiosis,Jouko1979,Makowsky1976} for more details on the $\DDelta$-operation and its use.

\begin{definition}  \label{DefDelta} Let $\L^{*}$ be a logic and let $\tau$  a fixed vocabulary. A class $\K$ of $\tau$-structures is \emph{$\SSigma(\L^{*})$-axiomatisable} if there exists $\phi$ in some extended vocabulary  $\tau' \supseteq \tau$ such that   $$\K = \{\A \, : \, \exists \B \; ( \B \models \phi  \text{ and }  \A = \B {\till} \tau )\}.$$ We say that $\K$ is the \emph{projection} of the class $\Mod(\phi)$ to $\tau$.

\p A class $\K$ is \emph{$\PPi(\L^{*})$-axiomatisable} if the complement of $\K$ (i.e., the class of $\tau$-structures not in $\K$) is $\SSigma(\L^*)$-axiomatizable, and $\DDelta(\L^*)$-axiomatizable if it is both $\SSigma(\L^*)$ and $\PPi(\L^*)$-axiomatizable. %are $\SSigma(\L^{*})$-axiomatisable.
 \end{definition}

Note that, if $\tau'$ has more sorts than $\tau$, then the structures $\B$ can be larger than their reducts $\A = \B {\till} \tau$. % in the above definitions.

Since $\DDelta(\L^*)$-axiomatizable classes   are closed under unions, intersections, complements and projections, one could consider $\DDelta(\L^*)$ itself as an abstract logic, whose model classes are exactly the  $\DDelta(\L^*)$-axiomatizable classes. In general,  $\DDelta(\L^*)$  is a non-trivial extension of $\L^*$.  However, for first-order logic, and in general any logic satisfying  the  the \emph{Craig Interpolation Theorem}, two notions coincide, see \cite[Lemma 2.7]{Makowsky1976}.

%\bigskip 
%With this, we can define \emph{Symbiosis}:

\begin{definition}[\textbf{Symbiosis}]\label{def:symbiosis} Let $\L^{*}$ be a logic and $R$  a predicate in the language of set theory. Then we say that \emph{$\L^{*}$ and $R$ are symbiotic} if: \begin{enumerate}[(1)]
%\item Every for every vocabulary $\tau$, every $\L^{*}$-axiomatisable class of $\tau$-structures is $\Delta_1(R)$-definable with parameter $\tau$, and
\item the relation $\models_{\L^{*}}$ is $\Delta_1(R)$-definable, and
\item for every finite vocabulary $\tau$, every $\Delta_1(R)$-definable class of $\tau$-structures closed under isomorphisms is $\DDelta(\L^{*})$-axiomatisable. \end{enumerate} \end{definition}

%The previous definition as well as all the notions of symbiosis we will define later can easily be adapted to a finite set of predicates. For notational simplicity we will only consider the case of symbiosis between a logic and a single predicate. 

In \cite{Symbiosis}, \emph{symbiosis} was established for many logic-predicate pairs, among them the ones mentioned in the introduction. 

 In practice, there is an equivalent condition to (2) which is easier to both verify and to apply. Let $R$ be an $n$-ary predicate in the language of set theory. We say that a transitive model of set theory $M$ is \emph{$R$-correct} if for all $m_1,\ldots,m_n \in M$ we have $M \models R(m_1,\ldots,m_n)$ iff $R(m_1,\ldots,m_n )$. %Note that the statement ``$M$ is {$R$-correct}'' is $\Delta_1(R)$. %-definable. %also $\Pi_1$.

\begin{lemma}[V\"a\"an\"anen \cite{Jouko1979}] \label{QR} For any predicate $R$ and logic $\L^*$, the following are equivalent: 
 \begin{enumerate}[(a)]
\item For every finite $\tau$, every $\Delta_1(R)$ class of $\tau$-structures closed under isomorphisms is $\DDelta(\L^{*})$-axiomatisable. 
%\item The class  $\Q_R := \{\A   \,: \, \A = (A,E,a_1,\ldots,a_n)$ is isomorphic to a transitive model $(M,\in,m_1,\ldots,m_n)$ and $R(m_1,\ldots,m_n)$ holds$\}$ is $\DDelta(\L^{*})$-axiomatisable.
\item The class ${\Q}_R := \{\A   \,: \, \A$ is isomorphic to a transitive $R$-correct $\in$-model$\}$ is $\DDelta(\L^{*})$-axiomatisable.
\end{enumerate}
\end{lemma}

\begin{proof} See  \cite{Jouko1979}, or a simpler version of  Theorem \ref{lemmacombi}. %The statements 2 and 3 express the same thing and we leave the proof of their equivalence to the reader. % but  essentially the same, but sometimes one is easier to use than the other. 
 \end{proof}

For completeness, and to illustrate how proofs of \emph{symbiosis} typically work in view of the results in the next section, we will now sketch  proofs of some paradigmatic examples  (see Section \ref{SectionIntro} for the definitions). $\ZFC^{-*}$ refers to a sufficiently large fragment of $\ZFC - $Power Set.

\newcommand{\LWF}{{\L_{\mathsf{WF}}}}

\begin{proposition}[V\"a\"an\"anen \cite{Jouko1979}] \label{symbiosis1} $\LWF$ and $\varnothing$ (no predicates) are symbiotic. \end{proposition}

\begin{proof}

\begin{enumerate}[(1)]
\item Since the statement ``$(A,E)$ is well-founded'' for sets $A$ is $\Delta_1$ and therefore absolute for transitive models,  ``$\A \models_\LWF \phi$'' is also absolute for transitive models. Then $\A \models_\LWF \phi $

\medskip
{ $\;\;\;\; \text{ iff }  \exists M \: (M \text{ transitive} \; \land \; M \models \ZFC^{-*} \; \land \;  \A \in M  \; \land \;       M \models (\A \models_\LWF \phi) )$}

\medskip
$\;\;\;\; \text{ iff }  \forall M \: (( M \text{ transitive} \; \land \;   M \models \ZFC^{-*} \; \land \;  \A \in M )  \; \to \; M \models (\A \models_\LWF \phi) )$.

\smallskip
This gives a   $\Delta_1$-definition.

%\bigskip Since the classical satisfaction relation is  $\Delta_1$ it follows that $\models_\LWF$ is $\Delta_1$ as well.

\item There are no predicates so  $\Q_\varnothing = \{(A,E)  :  (A,E)$ is isomorphic to a transitive $\in$-model$\}$. But  $(A,E)$ is isomorphic to a transitive $\in$-model iff $E$ is well-founded and extensional. Therefore  %$(A,E)$ is isomorphic to a transitive $\in$-model iff $E$ is well-founded and extensional. Therefore
 $(A,E)\in Q_\varnothing$ iff
$$(A,E) \models \textsc{Ext}\;  \land \;\mathsf{WF} xy (xEy)$$
which is an $\LWF$-sentence. Thus $\Q_\varnothing$ is $\LWF$-axiomatizable and therefore also $\DDelta(\LWF)$-axiomatizable. \qedhere
\end{enumerate}
\end{proof}

\begin{proposition}[V\"a\"an\"anen \cite{Jouko1979}] \label{symbiosis2} $\LI$ and $\Cd$ are symbiotic. \end{proposition}

\begin{proof} $\;$ \ 

\begin{enumerate}[(1)]
\item It is easy to see that  ``$\A \models_\LI \phi$'' is  absolute for  models of set theory which are $\Cd$-correct. Therefore $\A \models_\LI \phi $

\medskip
  $\;\; \text{ iff }  \exists M \: (M \text{ trans. and $\Cd$-correct}\; \land \; M \models \ZFC^{-*}\; \land \; \A \in M \; \land \;  M \models (\A \models_\LI \phi) )$

\medskip
 $\;\; \text{ iff }  \forall M \: (( M \text{ trans.  and $\Cd$-correct} \land   M \models \ZFC^{-*}   \land  \A \in M)  \; \to $

$\; M \models (\A \models_\LI \phi) )$

\smallskip
Note that ``$M$ is $\Cd$-correct'' is the statement $\forall x \in M \: ( (M \models \Cd(x) ) \leqqq \Cd(x))$ which is $\Delta_1(\Cd)$. Thus the above is a $\Delta_1(\Cd)$ formula.

\item We  need to check that $\Q_\Cd = \{(A,E) \; : \; (A,E)$ is isomorphic to a transitive $\Cd$-correct $\in$-model$\}$ is $\DDelta(\LI)$. We have:

$(A,E) \in \Q_\Cd \text{ iff }$
\begin{enumerate}
\item $E$ is wellfounded
\item $ (A,E) \models \textsc{Ext}$
\item $ (A,E) \models_\LI $ ``$\forall \alpha \: ( \Cd(\alpha) \to \forall x \in \alpha \lnot Iyz(y \in x)(y \in \alpha))$''  (written in $E$ instead of $\in$). \end{enumerate}
Conditions (b) and (c) are $\LI$-sentences, so it remains to show that (a) is $\DDelta(\LI)$. 

\p  First, we add a new unary predicate symbol $P$ and consider the sentence ``$P$ has no $E$-least element'', i.e., $$\phi \equiv \forall x ( P(x) \to \exists y (P(y) \land y E x)).$$ Clearly  the class of all models $(A,E)$ such that $E$ is \emph{not} well-founded is  the projection of $\Mod(\phi)$ to $\{E\}$. Therefore (a) is $\PPi(\LI)$.

\p To show that (a) is also $\SSigma(\LI)$ we  use a  trick due to Per Lindstr\"om \cite{Lindstrom1966}:  $(A,E)$ is well-founded if and only if we can associate    sets $X_a$  to every $a \in A$ in such a way that $a E b \: \to \: |X_a| < |X_b|$. Let the original sort be called $s_0$, extend the language with a second sort $s_1$,  add a new binary relation symbol $R$ from $s_0$ to $s_1$, and consider the sentence
$$\phi \; \equiv \; \forall^0  a \: \forall^0 b \: ( a E b \: \to \: ( \forall^1 x (R(a,x) \to R(b,x)) \;\land \; \lnot I y z R(a,y) R(b,z) ) ) $$
where we have used $\forall^0$ and $\forall^1$ to denote quantification over the two sorts. 

\p Now we can easily see that if $\mathcal{A} = (A, X, E^\A, R^\A) \models \phi$, then the sets $X_a :=  \{x \in X :  R^\A(a,x)\}$ are exactly as required, hence $E^\M$ is well-founded. Conversely, if $(A,E^\A)$ is well-founded, let $\rk_E: A \to \Ord$ be the rank function induced by $E^\A$,  let $X := \aleph_{\sup_{a \in A}\{\rk_E(a)+1\}}$, and define $R^\A \subseteq A \times X$  by $R^\A(a,\alpha) \; \Leftrightarrow \; \alpha < \aleph_{\rk_E(a)}$. Then $(A,X,E^\A,R^\A) \models \phi$.
 
 \p Thus we conclude that the class of well-founded structures is the projection of $\Mod(\phi)$ to $s_0$, completing the proof. \qedhere

\end{enumerate}
\end{proof}

\begin{proposition}[V\"a\"an\"anen \cite{Jouko1979}] \label{symbiosis3} $\L^2$ and $\PwSt$ are symbiotic. \end{proposition}

\begin{proof}  \begin{enumerate}[(1)]

\item As before, note that the statement `` $\A \models_{\L^2} \phi$'' is absolute for transitive models which are $\PwSt$-correct, so we have  $\A \models_{\L^2} \phi$

\medskip
$\;\;\;\; \text{ iff }  \exists M \: (M \text{ trans. and $\PwSt$-correct} \; \land \;  M \models \ZFC^{-*} \; \land \; \A \in M \; \land \;  M \models (\A \models_\LI \phi) )$

\medskip
$\;\;\;\; \text{ iff }  \forall M \: (( M \text{ trans.  and $\PwSt$-correct}  \; \land \;   M \models \ZFC^{-*} \; \land \;  \A \in M  )  \; \to \; M \models (\A \models_\LI \phi) )$

%\medskip
%$\;\;\;\; \text{ asdasd }  \exists \Ve_\alpha \: (\rho(\A) < \alpha \; \land \; \Ve_\alpha \models (\A \models_{\L^2} \phi) )$

%\medskip
%$\;\;\;\; \text{ asdasdas }  \forall  \Ve_\alpha \:  \: ( \rho(\A) < \alpha   \; \to \; \Ve_\alpha \models (\A \models_{\L^2} \phi) ).$

%\p This gives a $\Delta_1(\PwSt)$-statement as before.

\item Consider $\Q_\PwSt = \{(A,E)  :   (A,E)$ is isomorphic to a $\PwSt$-correct $\in$-model$\}$.  We have $(A,E) \in \Q_\PwSt$ iff $E$ is wellfounded and extensional and $(A,E) \models_{\L^2} ( y = \Pow(x) \: \leqqq \; \Phi(x,y))$ where $\Phi(x,y)$ is the $\L^2$-formula expressing that $y$ is the \emph{true} power set of $x$, written using $E$ instead of $\in$.\footnote{ If we use superscripts $0$ and $1$ to denote first- and second-order quantification, and  the relation symbols $\in^{00}$ and $\in^{01}$ to denote sets-in-sets membership and sets-in-classes membership, respectively, the sentence $\Phi(x,y)$ can be written as follows:
$$ \forall^1 Z (  \exists^0 v (v \in^{00} y \land \forall^0 w \: (w \in^{00} v \leqqq w \in^{01} Z)) \;\; \leqqq \;\; \forall^0 v( v \in^{01} Z \to v \in^{00} y) )$$}  All of this is expressible in $\L^2$.  \qedhere

\end{enumerate}
\end{proof}

\section{Bounded Symbiosis} \label{SecBoundedSymbiosis}

Although \emph{symbiosis} is stated as a property of  $\L^{*}$, it is really a property of $\DDelta(\L^{*})$. For many applications, this is irrelevant: for example, the \emph{downwards} L\"owenheim-Skolem principles are all preserved by the $\DDelta$-operation. However, in \cite[Theorem 4.1]{VaananenBounded} it was shown that the Hanf-number may not be preserved, and the  \emph{bounded} $\DDelta$-operation was introduced as a closely related operation  which still fulfills most of the properties  but, in addition, preserves Hanf-numbers. The bounded $\DDelta$ coincides with the original $\DDelta$ in many but not all cases, see \cite{VaananenBounded}.

\bigskip If we want to apply \emph{symbiosis} to upwards L\"owenheim-Skolem principles,  we need to accommodate this bounded version of the $\DDelta$-operation. Unfortunately, this also requires adapting the set-theoretic complexity classes to bounded versions.  This section is devoted to the definition of these concepts. We start with the model-theoretic side. The following definition generalizes Definition \ref{DefDelta} and was first introduced in \cite[p.\,45]{VaananenBounded}.

\begin{definition}  \label{bounded} A class $\K$ of $\tau$-structures is \emph{$\SSigma^{\mathrm{B}}(\L^{*})$-axiomatisable} if there exists $\phi$ in an extended vocabulary $\tau' \supseteq \tau$ such that  $$\K = \{\A \,: \, \exists \B \; ( \B \models \phi  \text{ and }  \A = \B {\till} \tau )\},$$  and for all $\A$ there exists a cardinal $\lambda_\A$, such that for any  $\tau'$-structure $\B$: if $\B \models \phi$ and $\A = \B {\till} \tau$ then $|\B| \leq \lambda_\A$.   

\p We say that $\K$ is a \emph{bounded projection} of $\Mod(\phi)$. $\K$ is \emph{$\DDelta^{\mathrm{B}}(\L^{*})$-axiomatisable} if both $\K$ and its complement  are $\SSigma^{\mathrm{B}}(\L^{*})$-axiomatisable. \end{definition}

In other words, $\K$ is a {bounded projection} of  $\Mod(\phi)$ if it is a projection and, in addition, every structure $\B \in \Mod(\phi)$ is bounded in its cardinality by a function that depends on the respective reduct $\B \till \tau$. Note that this definition really only plays a role when the extended vocabulary has additional sorts, since otherwise the cardinalities of $\B {\till} \tau$ and $\B$ are  the  same.  

Typical examples of  bounded projection will be seen, e.g.,  in Propositions \ref{boundedsymbiosis1}, \ref{nontrivial} and \ref{boundedsymbiosis3}.

It will be useful to define a bound given by a function from ordinals to ordinals rather than models to ordinals.

\begin{lemma}\label{Lemma:FunctionHBounded}
Suppose $\K$ is $\SSigma^{\mathrm{B}}(\L^{*})$-axiomatisable. Then there  exists $\phi$ in an extended vocabulary $\tau' \supseteq \tau$, and a non-decreasing function $h: \Ord\rightarrow \Ord$ such that  
 $$\K = \{\A \,: \, \exists \B \; ( \B \models \phi  \text{ and }  \A = \B {\till} \tau )\}$$
 and 
$$\forall  \B (\B \models \phi   \;\to\; |\B|\leq h(|\B\till \tau|)).$$
\end{lemma}
\begin{proof}  Define $h$ by $h(\lambda) := \sup\{\lambda_{\A}\,: \, |\A| \leq \lambda \}$ where each $\lambda_\A$ is as in Definitions \ref{bounded}. Since there are only set-many non-isomorphic models of any cardinality, $h$ is well-defined. \end{proof}

%Note that since there are only set-many non-isomorphic models of any cardinality $h$ is well-defined. Moreover it is an easy to see that $h$ is indeed non-decreasing and has the desired property.
%\end{proof}

\bigskip
Now we move to the set-theoretic side of things, which is more involved. %Since we have restricted the concept of   \emph{projection}  for abstract logics to a bounded version, this requires a corresponding change on the \emph{set-theoretic definability} side of things as well. 
In particular, we may no longer refer to arbitrary $\Sigma_1$-formulas, since   the witness in such formula may be unbounded, making it impossible to establish a \emph{symbiotic} relationship for bounded projective classes.  So we would like to restrict attention to formulas $\phi(x)$ of the form $\exists y \: \psi(x,y)$ but where, in addition, the (hereditary) size of at least one witness $y$ is bounded by a function of the (hereditary) size of $x$, and this function itself can  be ``captured'' by first-order logic. Note that a similar concept was introduced by the third author in  \cite[Definition 3.1]{Jouko1979}.

We first need to introduce the concept of ``being captured by first-order logic''.

%In the following we will denote tuples of variables of arbitrary length by $\vec{x}$. Similarly, given a set $A$, we will denote by $\vec{a}$ a tuple in $A$. 

%\begin{definition}
 %Given a $\Pi_1$ predicate $R$, a formula $\psi(x_1,\dots,x_n)$ is $\Sigma^{\mathrm{SB}}_1(R)$ if there is a $\Delta_0(R)$ formula $\phi(x_1,\ldots,x_n,y)$ such that:  $$\forall x_1,\ldots,x_n\;( \psi(x_1,\ldots,x_n)\leftrightarrow \exists y \; (\rho_\He(y)\leq \rho_\He(x_1,\ldots,x_n) \land \phi(x_1,\ldots,x_n,y)))$$ where $\rho_\He(x_1,\ldots,x_n):=\max\{\aleph_0,|\TC(\{x_1,\ldots,x_n\})|\}$. 
 %A formula is $\Pi_1^\mathrm{SB}(R)$ if its negation is $\Sigma_1^\mathrm{SB}(R)$, and $\Delta_1^{\mathrm{SB}}(R)$ if it is both $\Sigma^\mathrm{SB}_1(R)$ and $\Pi^\mathrm{SB}_1(R)$.
%\end{definition}

\begin{definition} \label{definablybounding}
A non-decreasing class function $F: {\rm Card} \rightarrow {\rm Card}$ %satisfying $F(\mu)\geq\max\{2^{\aleph_0}, 2^\mu\}$  
is called \emph{definably bounding} if the class of structures

$$\K:=\{(A,B)\, : \, |B| \leq F(|A|)\}$$

\noindent (in the vocabulary with  two sorts and no  symbols) is $\SSigma^{\mathrm{B}}(\L_{\omega\omega})$-axiomatisable.
\end{definition}

The intuition here is that the size of $|B|$ (the ``witness'') may be larger than $|A|$, but not by too much --- and by exactly how much is determined by $F$. For example, the identity function $F = $ id is definably bounding since we can always extend the language with a new function symbol $f$ between the two sorts and express   ``$f$ is a surjection'' in $\Lww$. A more interesting  example is the following:

\begin{Example} \label{alephbound} The function $F(\kappa)  = 2^\kappa$ is definably bounding.

\end{Example}

\begin{proof} Consider the class $\K:=\{(A,B)\, : \, |B| \leq 2^{|A|} \}$. Extend the vocabulary with a new relation symbol $E$ between the two sorts in reverse order, and consider the first-order formula %class $\K^{*} = \{$ of structures in the vocabulary $\tau_0:=\tau\cup \{R_0\}$ where $R_0\subset A\times B$ is a binary relation satisfying the $\L_{\omega\omega}$ formula:

$$\phi \; \equiv\; \forall^B b,b' (\forall^A a  (a E b  \leftrightarrow a E b') \rightarrow b=b'),$$
where we used the notation $\forall^A$ and $\forall^B$ to informally refer to quantification over the sorts. It is easy to see that if $\M = (A,B,E^\M) \models \phi$ then the map $i:B \to \Pow(A)$ given by $i(b) := \{a \in A \; : \; a E^\M b\}$ is injective, so $|B| \leq |\Pow(A)|$. It follows that $\K$ is the projection of $\Mod(\phi)$. The ``bounded'' part is immediate since we have not added new sorts.
\end{proof}

By an additional argument (see \cite[Lemma 6.23]{GaleottiThesis}), it is not hard to prove that if $F$ is definably bounding, then so is any iteration  $F^n$. In particular, if we define the cardinal function $\beth_n$ for infinite cardinals $\lambda$ by setting $\beth_0(\lambda) := 2^\lambda$ and $\beth_{n+1}(\lambda) := 2^{\beth_n(\lambda)}$, %(in other words, $\beth_n(\lambda) = 2^{2^{ .^{.^\lambda}}}$),
 then each such $\beth_n$ is also definably bounding. This is typically strong enough for most interesting applications.

%We can now introduce the set-theoretic counterpart to the model-theoretic  $\DDelta^B$-operator.

\begin{definition} $\;$

 \begin{enumerate}

\item For a set $x$, the \emph{$\He$-rank} of $x$, denoted by $\rho_\He(x)$, is the least infinite $\kappa$ such that $x \in \He_{\kappa^+}$ (i.e., $\rho_\He(\kappa) = \min(\aleph_0, |\trcl(x)|)$).

\item Let $F$ be a definably bounding function. A set-theoretic formula $\phi(x)$ is $\Sigma_1^F$ if there exists a $\Delta_0$ formula $\psi(x,y)$  such that  \begin{enumerate}
\item $\forall x \; (\phi(x) \: \leqqq \: \exists y \: \psi(x,y))$, and
\item $\forall x \;(\phi(x) \to \exists y' ( \rho_\He(y') \leq F(\rho_\He(x)) \;\;  \land \;\;  \psi(x,y') )$ \end{enumerate} 

%$$\forall x \; (\phi(x) \;\; \leqqq \;\;  \exists y \; \psi(x,y) \;\;  \leqqq \;\; \exists y' ( \rho_\He(y') \leq F(\rho_\He(x)) \;\;  \land \;\;  \psi(x,y') )$$
% {\color{blue} QUESTION: is this  the correct form? Or do we simply want $$\forall x \; (\phi(x) \: \leqqq \: \exists y ( \rho_\He(y) \leq F(\rho_\He(x)) \;\;  \land \;\;  \psi(x,y) )$$ The problem with this is that we cannot conclude from $M \models \psi(x,y)$ and $x,y \in M$, that $M \models \phi(x)$.  I am not sure if this is a problem or not.)
%}

 A formula is $\Pi_1^F$ if its negation is $\Sigma_1^F$, and $\Delta_1^F$ if it is equivalent to both a $\Sigma_1^F$- and a $\Pi_1^F$-formula. 
\item Let $R$ be a predicate in the language of set theory. All of the above can be generalized to $\Sigma_1^F(R)$, $\Pi_1^F(R)$ and $\Delta_1^F(R)$ in the obvious way. % by allowing $\Delta_0(R)$-formulas $\psi$.

\end{enumerate}
\end{definition}

So, a  $\Sigma_1^F$ formula is a $\Sigma_1$ formula such that, in addition, at least one   ``witness'' $y$  is not too far up in terms of $\He$-rank in relation to  $x$ itself, where by ``not too far up'' we mean ``bounded by the definably bounding function $F$''. An important example is the satisfaction relation of first-order logic:

\begin{Remark} \label{remarklww} The satisfaction relation $\models_\Lww$ is $\Delta_1^\id$ (see \cite{BarwiseAdmissibleSets}).
\end{Remark}

This leads us to introduce a new notion of symbiosis. A similar idea already appeared in   \cite[Definition 3.3]{Jouko1979}.

\begin{definition}[\textbf{Bounded Symbiosis}] Let $\L^{*}$ be a logic and $R$ a set theoretic predicate. We say that $\L^{*}$ and $R$ are \emph{boundedly symbiotic} if \begin{enumerate}[(1)]
\item % Every for every vocabulary $\tau$, every $\L^{*}$-axiomatisable class of $\tau$-structures is $\Delta^{\mathrm{SB}}_1(R)$-definable with parameter $\tau$, and
The  relation $\models_{\L^{*}}$ is $\Delta^F_1(R)$-definable for some definably bounding $F$.
\item  Every $\Delta_1^F(R)$-definable class of $\tau$-structures closed under isomorphisms is $\DDelta^{\rm B}(\L^*)$-axiomatizable (for every definably bounding $F$).
%\item the class  $\overline{\Q}_R := \{\A   \,;\, \A = (A,E)\cong (M,\in)$ with $M$ transitive and $R$-correct$\}$ is $\Delta^{\mathrm{B}}(\L^{*})$-axiomatisable.
\end{enumerate} \end{definition}
 
Just as before, condition (2) of \emph{bounded symbiosis} has an equivalent form which is usually easier to verify and to apply. A side effect will be that in (2), we may assume that $F = \id$ without loss of generality. First, we need the following:

\begin{lemma} \label{Cor:ABSPI1} \label{lemmichka}  Let $R$ be a $\Pi_1$ predicate in set theory. \begin{enumerate}
\item Every $\He_\kappa$ is $R$-correct.
\item Let    $\phi$ be a $\Sigma_1^{{F}}(R)$-formula.  Then for every $x$ and every $\kappa > F(\rho_\He(x))$,  $\phi(x)$ is absolute (upwards and downwards) for  $\He_{\kappa}$. 
\end{enumerate}
% with $\kappa \geq  F(|{\rm trcl}(x)|^+)$. 
%\item For any transitive set $A$, there exists an $\He_\theta$   such that $\rho_\He(\He_\theta) \leq 2^{2^{\rho_\He(A)}}$.

\end{lemma}

\begin{proof} 1 is a classical result of  of L\`evy \cite{Levy} (see also \cite[Lemma 6.27]{GaleottiThesis}). From this, it follows that $\Delta_0(R)$-formulas are  absolute for $\He_\kappa$. 

\p For 2, it suffices to prove downwards absoluteness. Let $x$ be arbitrary and suppose $\phi(x)$ holds. Then there exists $y$ such that $\rho_\He(y) \leq F(\rho_\He(x)) < \kappa$ and   $\psi(x,y)$ holds, where   $\psi(x,y)$ is  the corresponding $\Delta_0(R)$-formula.  But then  $y \in \He_{\kappa}$ and $\He_{\kappa} \models \psi(x,y)$ by the above. It follows that  $\He_{\kappa} \models \phi(x)$. \end{proof}

\begin{lemma}\label{Lemma:EqBoundDef} \label{lemmacombi}
Let $\L^{*}$ be a logic and $R$ be a $\Pi_1$ predicate. Then the following are equivalent:
\begin{enumerate}[(a)]
\item Every $\Delta^{{F}}_1(R)$ class of $\tau$-structures closed under isomorphisms is  $\DDelta^{{\rm B}}(\L^{*})$-axiomatisable (for every definably bounding $F$).
\item Every $\Delta^{{\id}}_1(R)$ class of $\tau$-structures closed under isomorphisms is  $\DDelta^{{\rm B}}(\L^{*})$-axiomatisable.

%{\color{blue} every $\Sigma^{{B}}_1(R)$ class of $\tau$-structures closed under isomorphisms is $\SSigma^{{B}}(\L^{*})$-axiomatisable.}

% \item  The class  $\Q_R := \{\A   \,: \, \A = (A,E,a_1,\ldots,a_n)$ is isomorphic to a transitive model $(M,\in,m_1,\ldots,m_n)$ and $R(m_1,\ldots,m_n)$ holds$\}$ is $\Delta^{{B}}(\L^{*})$-axiomatisable,
\item The class ${\Q}_R := \{\A   \,: \, \A$ is isomorphic to a transitive $R$-correct $\in$-model$\}$ is $\DDelta^{{B}}(\L^{*})$-axiomatisable.

\end{enumerate} 
\end{lemma}
\begin{proof} $(a) \Rightarrow (b)$ is immediate. For $(b) \Rightarrow (c)$, it is enough to prove that $\Q_R$ itself is  $\Delta^{\id}_1(R)$-definable. We have $\A\in \Q_R$ iff $\exists M \: \exists f$ such that

\begin{enumerate}
\item $\rho_\He(M) \leq \rho_\He(\A)$
\item $\rho_\He(f) \leq \rho_\He(\A)$
\item $M$ is transitive
\item $f: \A = (A,E) \cong (M, \in)$ is an isomorphism
\item $\forall x_1 \dots x_n \in M \; (M \models R(x_1, \dots, x_n) \;\leqqq\; R(x_1, \dots, x_n))$
\end{enumerate}

\p Since clauses 3--5 are $\Delta_1(R)$, this gives a  $\Sigma^\id_1(R)$ statement. Similarly,  $\A \notin \Q_R$ iff $(A,E)$ is not well-founded or  $\exists M \: \exists f$ such that

\begin{enumerate}
\item $\rho_\He(M) \leq \rho_\He(\A)$ 
\item $\rho_\He(f) \leq \rho_\He(\A)$
\item $M$ is transitive
\item $f: \A = (A,E) \cong (M, \in)$ is an isomorphism
\item $\lnot \forall x_1 \dots x_n \in M \; (M \models R(x_1, \dots, x_n) \;\leqqq\; R(x_1, \dots, x_n))$
\end{enumerate}

\p It is easy to see that being ill-founded is $\Sigma_1^{\id}$, so the conjunction is again a $\Sigma^{\id}_1(R)$ statement.

\bigskip \noindent Now we look at $(c) \Rightarrow (a)$.  Let $\K$ be a $\Delta^{F}_1(R)$-definable class over a fixed vocabulary $\tau$ which is closed under isomorphisms. Let $\Phi(x)$ be the $\Sigma^{F}_1(R)$ formula defining $\K$. For simplicity, assume that $\tau$ consists only of one binary predicate $P$ and only one sort.

\p Let $\tau'$ be a language in two sorts $s_0$ and $s_1$, with $E$ a binary relation symbol of sort  $s_0$, $G$ a function symbol from $s_1$ to $s_0$, $c$ a constant symbol of sort $s_0$, and $P$ a unary predicate symbol in $s_1$ (i.e., $s_1$ is the original sort of $\tau$, while $s_0$  adds a ``model of set theory'' on the side). 

\p Let $\K'$ be the class of all $\tau'$-structures 
$$\mathcal{M} := \left(M,A,E^\M,c^\M, G^\M, P^\M\right)$$
satisfying the following conditions: \begin{enumerate}
\item \label{Lem:EqBSBpoint2} $(M,E^\M)\in {\Q}_R$, i.e., $(M,E^\M)$ is isomorphic to a transitive $R$-correct model
\item \label{Lem:EqBSBpoint1} $(M,E^\M)\models \ZFC^{-*}$ %, for a sufficiently large fragment of $\ZFC$ (e.g., we can take $\ZFC^{-}_n$ for a sufficiently large $n$)
\item \label{Lem:EqBSBpoint3} $\mathcal{M} \models \Phi(c)$
\item \label{Lem:EqBSBpoint4} $|M|\leq 2^{2^{F(|A|)}}$
\item \label{Lem:EqBSBpoint5} $\mathcal{M} \models \text{\vir {$c=(a,b)$ and $b\subseteq a\times a$}}$ (written using $E$ instead of $\in$)
\item \label{Lem:EqBSBpoint6} $\mathcal{M} \models $ ``$G$ is an isomorphism between $(A,P)$ and $(a,b)^{(M,E)}$''. In this sentence, $(a,b)^{(M,E)}$ refers to the domain and binary relation on it which is described by $a$ and $b$ when interpreting $\in$ by $E^\M$ (e.g., the domain is really $\{x \in M \; : \;   x E^\M a\}$ etc. )

\end{enumerate}
Now we can see that conditions 2, 3, 5 and 6 are directly expressible in $\Lww$, while 1 is $\DDelta^{\mathrm{B}}(\L^{*})$-axiomatisable, and hence $\SSigma^{\mathrm{B}}(\L^{*})$-axiomatisable, by assumption. Moreover,  4 is also $\SSigma^{\mathrm{B}}(\L^{*})$-axiomatisable: this follows by the definition of ``definably bounding'', by Example \ref{alephbound}, and the discussion following it.

\p It remains to prove that $\K$ is a bounded projection of $\K'$ to $\tau$. Note that the ``bounded'' part is immediate due to 4. \begin{itemize}

\item First suppose  $\M =  \left(M,A,E^\M,c^\M, G^\M, P^\M\right)\in \K'$. We want to show that $(A, P^\M) \in \K$. By \ref{Lem:EqBSBpoint2} $(M,E^\M)$ is isomorphic to a transitive model $(\overline{M},\in)$ which is $R$-correct. Let $\overline{c^{\M}}$ be the image of $c^\M$ under this isomorphism. Then $(\overline{M},\in )\models \Phi(\overline{c^{\M}})$. Moreover, since $\overline{M}$ is $R$-correct and  $\Phi$ is $\Sigma_1(R)$, by upwards absoluteness we have $\Phi(\overline{c^{\M}})$, i.e. $\overline{c^\M}\in \K$. By \ref{Lem:EqBSBpoint6}  we have $\overline{c^{\M}}\cong \text{``}(a,b)^{(M,E)}\text{''}\cong (A,P^\M)$. Since $\K$ is closed under isomorphism, it follows that  $(A,P^\M)\in \K$. 
 
 \item Conversely, let $\A = (A,P^\A) \in \K$, i.e., $\Phi(\A)$ holds. We want to find a structure $\M\in \K'$ such that $\A=\M{\upharpoonright} \tau$. 
 
 \p The first  idea would be to find an  $\He_\theta$ which is sufficiently large to reflect $\Phi(\A)$  while still being small enough to satisfy condition (4). In general, however, the  transitive closure of $A$ might be significantly larger than $|A|$. So we first find a model $\bar{\A}$   which is isomorphic to $\A$ but whose domain is some cardinal $\mu$. Since $\K$ is closed under isomorphisms, $\bar{\A}$ is also in $\K$, i.e. $\Phi(\bar{A})$ also holds.
 
 \p   Note that in this case $P^{\bar{\A}} \subseteq \mu \times \mu$, in particular, $\trcl( \bar{\A} ) = \trcl ( (\mu,P^{\bar{\A}}) ) \subseteq \mu$, so $\rho_\He(\bar{\A}) = \mu$.\footnote{Even though we made an assumption to only consider the language $\tau$ with one binary relation symbol for the sake of clarity,  the same holds for any number of predicate or function symbols on a model with domain $\mu$.} Let $\theta := F(\mu)^+$. By Lemma \ref{lemmichka} (2) $\He_\theta \models \Phi(\bar{\A})$.
 
 \p Now let $\M = (\He_\theta, A, \in, \bar{\A}, g, P^\A)$, where $g$ is the isomorphism between $\A$ and $\bar{\A}$. Now it is not hard  to verify that all 6 conditions in the definition of $\K'$ are satisfied. In particular, 1  holds because of Lemma \ref{lemmichka} (1) and 4  because $$|\He_\theta| \leq 2^\theta = 2^{F(\mu)^+} \leq 2^{2^{F(\mu)}} = 2^{2^{F(|A|)}}.$$ Thus $\M \in \K'$ as we wanted.

\end{itemize}

\p This shows that $\K$ is a bounded projection of $\K'$ and therefore $\K$ is $\SSigma^{\rm B}(\L^*)$. Since $\K$ is also $\Delta_1^F(R)$, the same proof works for the complement of $\K$, showing that $\K$ is $\DDelta^{\rm B}(\L^*)$. \qedhere

\end{proof}

\section{Examples of Bounded Symbiosis} \label{SecExamples}

In general, all the pairs that are proved to be \emph{symbiotic} in \cite[Proposition 4]{Symbiosis} are also \emph{bounded symbiotic.}  For completeness, we now show how the proofs of Propositions \ref{symbiosis1}, \ref{symbiosis2} and \ref{symbiosis3} can be strengthened to prove bounded symbiosis. In particular, Proposition \ref{nontrivial} is a non-trivial result since  by \cite[\S\,4]{VaananenBounded} it is consistent that  $\DDelta(\LI) \neq \DDelta^{\mathrm{B}}(\LI)$. 

\begin{proposition} \label{boundedsymbiosis1} The pairs $\LWF$ and $\varnothing$ are bounded symbiotic. \end{proposition}

\begin{proof} The same proof as Proposition \ref{symbiosis1} works. For (1), note that  we may always use reflection to find $M$ such that $|M|  = |\trcl(A)|$.   This implies that $\models_{\LWF}$ is $\Delta_1^{\id}$. For part (2), $\Q_{\rm WF}$ is actually $\LWF$-definable, hence $\DDelta^{\rm B}(\LWF)$-definable. \end{proof}

\newcommand{\Inf}{{\mathrm{Inf}}}
\newcommand{\Like}{\mathrm{Like}}

\begin{proposition} \label{nontrivial} The pairs $\LI$ and $\Cd$ are bounded symbiotic. \end{proposition}

\begin{proof} Again we look at the proof of Proposition \ref{symbiosis2}. For (1), we have the stronger equivalence:  $\A \models_\LI \phi $

\medskip
  $\;\; \text{ iff }  \exists M \: (\rho_{\He}(M) \leq 2^{2^{\rho_\He(\A)}} \; \land \; M \text{ transitive and $\Cd$-correct}\; \land \; M \models \ZFC^{-*}$
  
  $\;\; \; \;\;\;\;\; \land \;  \A \in M \; \land \;  M \models (\A \models_\LI \phi) )$

\medskip
$\;\; \text{ iff }  \forall M \: (( (\rho_{\He}(M) \leq 2^{2^{\rho_\He(\A)}}  \; \land \; M \text{ transitive  and $\Cd$-correct} $

$\;\;\;\;\;\;\;\;\; \land \;  M \models \ZFC^{-*}   \land  \A \in M)  \; \to \; M \models (\A \models_\LI \phi) )$

\p As in the proof of Lemma \ref{lemmacombi}, we know that for any $\A$ we can let $\theta = |\trcl(\A)|^+$, so that
 
 $$|\He_\theta| \leq 2^\theta = 2^{\rho_{\He}(\A)^+} \leq 2^{2^{\rho_\He(\A)}}$$
  and $\He_\theta$ is $\Cd$-correct by Lemma \ref{lemmichka} (1). Thus, the relation $\models_\LI$ is $\Delta_1^F(\Cd)$ for the definably bounding function $F(\alpha)  = 2^{2^{\alpha}}$.

 \p
Now we check (2) of bounded symbiosis. Again, looking at the proof of Proposition \ref{symbiosis2}, we see that clauses (b) and (c) are $\LI$-sentences, and (a) is $\PPi^{\rm B}(\LI)$ since we do not need to add new sorts. The only  issue, then, is to prove that ``$(A,E)$ is well-founded'' is $\SSigma^{\rm B}(\LI)$, which is less trivial because the method described previously   does not yield an upper bound on the size of the second sort. So we need to adapt this method. The idea is to add a new linear ordering $(B,<)$ to the structure  $(A,E)$, and a function $f:A \to B$, such that $B$ plays the role of the appropriate cardinals $\aleph_\alpha$.

\p Suppose $(B,<)$ is a linear order. For $b\in B$  let $b{\downarrow} = \{b' \in B \: : \: b' < b\}$ denote the set of $<$-predecessors of $b$. We say that $b$ is \emph{cardinal-like} if for every $b' < b$  we have $| {b'{\downarrow}} | < | b {\downarrow}|$, and the set $B$ itself is \emph{cardinal-like} if for every $b \in B$ we have $|b {\downarrow}| < |B|$.

\p Consider the language with two sorts $s_0$ and $s_1$, a binary relation symbol $E$ in $s_0$, a binary relation $<$ in $s_1$, and a function symbol $f$ from $s_0$ to $s_1$. In the following proof, we will informally refer to the domains of the respective sorts by $A$ and $B$. % and use the notation $\forall a \in A$, $\forall b \in B$ etc. to denote quantification over the sorts if not clear from context. %(in the following proof, we will generally not specify the sorts when this is clear from context).

\p First define the following abbreviations:
$$\Inf(b)  \; \equiv \; \forall^B x < b  \: \mathrm{I} y z (y< b \land y\neq x) (z < b)$$
i.e., $b$ has infinitely many $<$-predecessors.
$$\Like(b) \; \equiv \; \forall^B b' < b  \; \lnot \mathrm{I} y z (y< b) (z< b')),$$
i.e., $b$ is cardinal-like.  Let $\Phi$ be the conjunction of the following $\LI$-formulas:

%\begin{definition}
%Let $(A,<_A)$ be partial order. Given an element $a$ of $A$ we will denote by ${a{\downarrow}}$ the set $\{a'\in A\:;\: a'<_A a\}$ of \emph{predecessors} of $a$ in $A$. Given a cardinal $\kappa$, an element $a\in A$ is said to be \emph{$\kappa$-like} if $|{a{\downarrow}}| = \kappa$ and for every $a'\in {a{\downarrow}}$ we have $| {a'{\downarrow}} | < \kappa$.  
%\end{definition}
%\begin{lemma}\label{Lemma:WellOrdSigmaB}
%The class $\K$ of well-orders $(A,<)$ is $\Sigma(\LI)$-axiomatisable by a class $\K'$ such that
%$$\forall (A,<)\in \K \forall M'\in \K' (M'{\upharpoonright}\{<\}=(A,<) \rightarrow |M'|\leq \aleph_{\mathrm{OT}(A,<)}).$$
%where $\mathrm{OT}(A,<)$ is the unique ordinal isomorphic to $(A,<)$. So the class of well-orders is $\Sigma^{\mathrm{B}}(\LI)$-axiomatisable.
%\end{lemma}
%\begin{proof}
%Consider the class $\K'$ of structures of type $(A,B,<_A,<_B,f),$ where $<_A\subset A\times A$, $<_B\subset B\times B$, and $f:A\rightarrow B$ is a function. First we define the following formulas:
%$$\mathrm{Inf}(x):=\forall b<_B x (\mathrm{I} y,z (y<_Bx \land y\neq b) (z<_B x)),$$
%i.e., $x$ has infinitely many predecessors.
%$$\mathrm{Like}(x):=\forall b<_B x \lnot (\mathrm{I} y,z (y<_Bx) (z<_B b)),$$
%i.e., $x$ is $|{x{\downarrow}}|$-like.

%Let $\varphi$ be the following $\LI$ conjunction of the following sentences:

\begin{enumerate}[(i)]

\item \label{1} $<$ is a linear order;

\item \label{2} $\forall^B b \;\lnot \mathrm{I} x y (x <  b)(y=y)$

i.e., ``$B$ is cardinal-like'';

\item \label{3} $\forall^B b ( \mathrm{Inf}(b)\rightarrow \exists^B b' \leq b (\mathrm{I}xy (x< b)(y< b') \land \mathrm{Like}(b')))$

i.e., ``no infinite cardinals are skipped'';

\item \label{4} $\forall^A a \forall^A a' \;(   a E a' \rightarrow f(a) < f(a'))$

  i.e., ``$f$ is order-preserving'';

\item \label{5} $\forall^A a (\mathrm{Inf}(f(a)) \land \mathrm{Like}(f(a)))$

 i.e., ``every $f(a)$ is infinite and cardinal-like'';

\item \label{6} $\forall^A a \: \forall^B b \: (b < f(a) \: \to \: \exists^A a' \: (a' E a \; \land \; b \leq f(a'))$

 i.e., ``every  $|f(a) { \downarrow}|$ is the least cardinal higher than $|f(a') {\downarrow}|$ for all $a' E a$'';
 
 \item \label{7}  $\forall^B b \:  \exists^A a \: ( b \leq f(a))$
 
 i.e., ``the range of $f$ is cofinal in $B$''.

\end{enumerate}

\p Now we prove several claims, which  together imply that ``$(A,E)$ is well-founded'' is $\SSigma^{\rm B}(\LI)$. For ease of notation we will identify the symbols $E, <$ and $f$ with their respective interpretations.

\begin{claim} \label{claim1} $(A,E)$ is well-founded iff $(A,B,E,<,f) \models_\LI \Phi$ for some $B, <$ and $f$. \end{claim}

\begin{proof} First, suppose $(A,E)$ is well-founded. Let $\rk_E$ be the rank function induced by $E$, let $B = \aleph_{\rk_E(A)}$, and let $f(a)=\aleph_{\rk_E(a)}$. Then it is easy to verify that $(A,B,E,<,f) \models_\LI \Phi$. Conversely, suppose $(A,B,E,<,f) \models \Phi$. Then for every $aEa'$ we have $|f(a){\downarrow}| < |f(a'){\downarrow}|$, as follows easily from the fact that $<$ is transitive, that $f$ is order-preserving, and that every $f(a)$ is cardinal-like. But then $E$ must be well-founded. \qedhere  $\;$ (Claim \ref{claim1})
\end{proof}

\begin{claim} \label{claim2} Suppose $(A,B,E,<,f) \models_\LI \Phi$. Then \begin{enumerate}

\item  For all $b \in B$ and all $\lambda < |b {\downarrow}|$, there exists $c < b$ such that $\lambda \leq |c {\downarrow} | < |b {\downarrow}|$.
\item For all $b \in B$ and all $\lambda < |b {\downarrow}|$, there exists $d < b$ such that $|d  {\downarrow}| = \lambda$.
\item For all  $\lambda < |B|$, there exists $d$ such that $|d  {\downarrow}| = \lambda$.
\end{enumerate}
\end{claim}

\begin{proof} $\;$ 

\begin{enumerate}
\item Let $b \in B$ and $\lambda < |b {\downarrow}|$. By  (\ref{3}) there is a $b' < b$ such $|b' {\downarrow}| = |b {\downarrow}|$ and $b'$  is cardinal-like. We claim that there is $c<b'$ such that $\lambda \leq |c {\downarrow}|$. Towards contradiction, suppose this is false. Let $\{c_\alpha : \alpha<|b' {\downarrow}|\}$ enumerate $b' {\downarrow}$ and consider the initial $\lambda$-sequence of this enumeration, i.e., $\{c_\alpha  :  \alpha<\lambda\}$. This sequence cannot be $<$-cofinal in $b'$, otherwise we would have $|b' {\downarrow}| = | \bigcup_{\alpha<\lambda} (c_\alpha {\downarrow} ) |  \leq \lambda \times \lambda = \lambda$, which is a contradiction. Therefore, there is $c < b'$ such that $\{c_{\alpha} : \alpha < \lambda\} \subseteq c { \downarrow}$. But then   $\lambda \leq |c{\downarrow}|$, also a contradiction.%. Moreover, since $b'$ was cardinal-like, $|c {\downarrow}| < |b' {\downarrow}| = |b {\downarrow}| $, so we are done.
\item First apply (1) to find $c_0 < b$ such that $\lambda \leq |c_0 {\downarrow}| < |b {\downarrow}|$. Apply again to find $c_1 < c_0$ such that $\lambda \leq |c_1 {\downarrow}| < |c_0 {\downarrow}|$, etc. By well-foundedness, this process will stop after finitely many steps and we will find $d<b$ such that $\lambda = |d {\downarrow}|$.

\item By an analogous argument as in (1) above, and using (\ref{2}), we first find $b \in B$ such that $\lambda \leq |b {\downarrow}|< |B|$. Then proceed as in (2).  \qedhere  (Claim \ref{claim2}).

\end{enumerate}
\end{proof}

\begin{claim} \label{claim3} Suppose $(A,B,E,<,f) \models_\LI \Phi$. Then $|A \cup B| \leq  \aleph_{\rk_E(A)}$.
\end{claim}

\begin{proof} We prove, by induction on $E$, that for all $a \in A:$

$$|f(a){\downarrow}| \leq \aleph_{\rk_E(a)}.$$

\p Suppose the above holds for all $a' E a$. Towards contradiction suppose $|f(a){\downarrow}| > \aleph_{\rk_E(a)}$. By Claim \ref{claim2} (2), we can find $d < f(a)$ such that $|d {\downarrow}| =  \aleph_{\rk_E(a)}$. By (\ref{6}), there exists $a' E a$ such that $d \leq f(a')$. But this means that 
$$\aleph_{\rk_E(a)} = |d {\downarrow}| \leq  |f(a'){\downarrow}| =    \aleph_{\rk_E(a')}$$
which is a contradiction since $\rk_E(a') < \rk_E(a)$. This competes the induction.

\p  Completing the proof requires repeating the above argument once more: if $|B| > \aleph_{\rk_E(A)}$, then by Claim \ref{claim2} (3) there is $d \in B$ such that $|d {\downarrow}| = \aleph_{\rk_E(A)}$, and by   (\ref{7}) there is $a \in A$ with $d \leq f(a)$, implying

$$\aleph_{\rk_E(A)} = |d {\downarrow}| \leq  |f(a){\downarrow}| =   \aleph_{\rk_E(a)}$$
 which is a contradiction since by definition $\rk_E(a) < \rk_E(A)$. It follows that $|A \cup B| = |B| \leq \aleph_{\rk_E(A)}$. \qedhere  $\;$ (Claim \ref{claim3})

%\bigskip Claims \ref{claim1} and \ref{claim3} together complete the proof.
\end{proof}

\end{proof}

\begin{proposition} \label{boundedsymbiosis3} The pairs $\L^2$ and $\PwSt$ are bounded symbiotic. \end{proposition}

 \begin{proof} A straightforward adaptation of the proof of Proposition \ref{symbiosis3} works. Using the same trick as above, in (1) we see that $\models_{\L^2}$ is $\Delta^F_1(\PwSt)$ for $F(\alpha) = 2^{2^{\alpha}}$.  For (2), we do not need to change anything since the class $\Q_\PwSt$ is already $\L^2$-axiomatisable.  \end{proof}

\section{The Upwards Structural Reflection principle}\label{SecCountable}

Now we consider  a \emph{reflection number} analogous to the one in Definition \ref{SR} which, as in \cite{Symbiosis}, will allow us to connect the strength of existence of upward L\"owenheim-Skolem numbers for strong logics to large cardinals. 

%Before we define these upward reflection principles let us note that in general, given a $\Pi_1$ relation $R$ in the language of set theory, $\Sigma_1^{B}(R)$ classes can be very ill-behaved. In particular if a $\Sigma_1^{B}(R)$ class is is such that there is $$

%\begin{definition} We say that a class $\K$ of structures in some vocabulary $\tau$ is \emph{transitively-represented} if
%$$\forall \A\in \K \; \exists \B\in \K (\A\cong \B ~\land~\text{the domains of $\B$ are transitive sets}).$$
% Note that every class closed under isomorphisms is transitively-represented.
%\end{definition}
%\noindent

%{\color{blue} \noindent COMMENT: it would be nice to  get rid of this assumption in the definition of $\USR$, or show that it is necessary.}

%In analogy to Definition  \ref{SR} we consider the following upwards structural reflection principle:

\begin{definition}  Let $R$ be a $\Pi_1$ predicate in the language of set theory. The \emph{bounded upwards structural reflection number}  $\USR(R)$ is the least $\kappa$ such that:

\begin{quote} For every definably bounding function $F$, and every $\Sigma^{F}_1(R)$-definable class of $\tau$-structures in a fixed vocabulary $\tau$ closed under isomorphisms:

 If there is $\A \in \K$ with $|\A| \geq \kappa$, then for every $\kappa' > \kappa$ there is a $\B \in \K$ with $|\B| \geq \kappa'$ and an elementary embedding $e:\A \embed_{\L_{\omega \omega}} \B$. \end{quote}

If there is no such cardinal,   $\USR(R)$ is undefined.

\end{definition}

\begin{Remark} In this definition, we are assuming that $\K$ is definable by a $\Sigma^{F}_1(R)$-formula \emph{without  parameters}. In particular, the definition presupposes that $\tau$, as a vocabulary, is itself $\Sigma^{F}_1(R)$-definable (e.g., finite). Notice that if arbitrary $\tau$ were allowed,  $\USR(R)$ would never be defined: for any $\kappa$, we could take a vocabulary $\tau$ with $\kappa$-many constant symbols and let $\K$ be the class of $\tau$-structures such that every element is the interpretation of a constant symbol, which is $\Delta_1^\id$ in $\tau$. One could avoid this problem by considering classes defined with parameters of a limited $\He$-rank; but then, to prove results like the ones in this section, one would need to extend the corresponding logic in such a way that the parameter can be defined.  For the current paper, the parameter-free version will be sufficient. \end{Remark}

Our main theorem below is proved for logics which have $\Delta_0$-definable syntax and dependence number $\omega$. This is necessary if we want to avoid parameters---recall the discussion in Section \ref{SectionPreliminaries}. All logics obtained by adding finitely many quantifiers to first- or second-order logic, such as $\LWF, \LI, \L^2$ and the examples in  \cite[Proposition 4]{Symbiosis}, are covered by this theorem. For the $\ULST$-principle, see Definition \ref{upwn} and recall  that since we are assuming $\dep(\L^*) = \omega$,   $\ULST_{\omega}(\L^{*}) = \ULST_{\infty}(\L^{*})$.

\begin{theorem}[Main Theorem] \label{SymbiosisTheorem}   
Let $\L^{*}$ be a logic with  $\Delta_0$-definable syntax and $\dep(\L^*) = \omega$, and let $R$ be a $\Pi_1$ predicate. Assume that $\L^{*}$ and $R$ are {boundedly symbiotic}. %, and $\L^{*}$ has $\mathrm{dep}(\L^{*})=\omega$ and is $\Delta_1^{\mathrm{B}}(R)$-finitely-definable. Moreover, let $\lambda$ be a cardinal such that there is a sequence $(\delta_n)_{n\in \omega}$ of $\Delta_1^{\mathrm{B}}(R)$-definable cardinals such that $\bigcup_{n\in \omega}\delta_n=\lambda$. 
Then the following are equivalent:

\begin{enumerate}[$\;\;\;\;\;(1)$]
\item $\ULST_{\infty}(\L^{*}) = \kappa$,  
\item $\USR(R) = \kappa$. 
\end{enumerate}
 
 \end{theorem}

\begin{proof} $(2) \Rightarrow (1)$. Suppose $\USR(R) = \kappa$.  Let  $\phi$ be an $\L^*$-formula, and let $\A \models_{\L^*} \phi$ with $|\A| \geq \kappa$. Letting $\kappa'$ be any cardinal above $\kappa$, the  goal is to find a super-structure $\B$ of $\A$ such that $\B \models_{\L^*} \phi$ and   $|\B| \geq \kappa'$.

\p Consider the class $\K = \Mod(\phi)$. By condition (1) of \emph{bounded symbiosis},  $\K$ is $\Delta_1^F(R)$-definable, hence $\Sigma_1^F(R)$, with parameter $\phi$. However, since $\dep(\L^*) = \omega$, we may assume  that the vocabulary of $\phi$ is finite. Moreover,   $\L^*$ has a $\Delta_0$-definable syntax, so $\phi$ is $\Delta_0$-definable, therefore $\K$ is in fact $\Sigma_1^F(R)$-definable \emph{without} parameters. It is also  clearly closed under isomorphisms. %Being closed under isomorphisms, $\K$ is  transitively-represented.

\p Applying  $\USR(R)$ we find a $\B' \in \K$, such that   $|\B'| \geq \kappa'$ and  there is $e: \A  \; {\embedFOL} \; \B'$. We can also easily find $\B \cong \B'$ such that $\A$ is a substructure of $\B$, and this is what we need.

\bigskip

 \p $(1) \Rightarrow (2)$. Now assume $ \ULST_\infty(\L^*) = \kappa$, and let $\K$ be a $\Sigma^{F}_1(R)$-definable transitive class of $\tau$-structures, with $\Phi(x)$ the defining $\Sigma_1^{F}(R)$-formula.% (\emph{without parameters}). 

  {%\definecolor{col}{rgb}{0.8, 0.18, 0.16} \color{col}    \textbf{Proof 2} (This proof follows the method of \cite{Symbiosis}, but we are not sure whether there is something wrong with the easier proof.  
  
  \p Since the $\USR$-principle involves elementary embeddings whereas $\ULST$ does not, the proof must proceed indirectly. The intuition is that we first embed a given structure $\A \in \K$ in a larger structure that includes a model of set theory $\He_\theta$ and includes Skolem functions for first-order existential sentences; then we apply $\ULST$ to (a further extension of) this larger structure. To make sure that enlarging the set-theoretic structure also yields an enlargement of the original structure, we must carefully keep track of the relations between cardinalities given by the various bounds in the definition of \emph{bounded symbiosis}. See Figure \ref{triangles}.

  \begin{figure}[h] 
  \begin{center}
  \vskip-0.3cm \includegraphics[width=13cm]{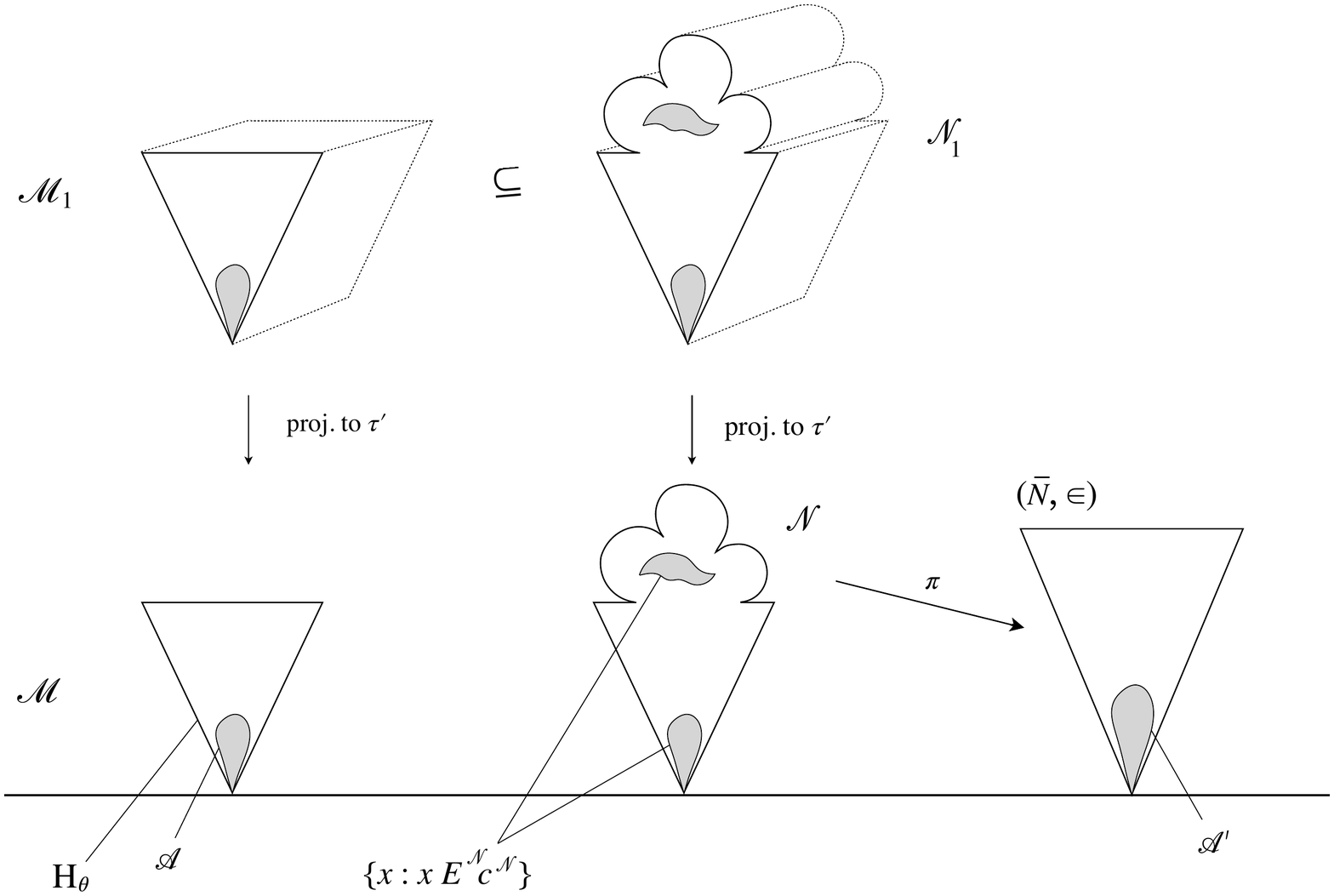} \vskip-0.5cm
  \end{center}
\caption{Structure of the proof.} \label{triangles} 
  \end{figure}

  \p Assume that  $\tau$ is in one sort (a similar proof works in the general case). Similarly to the proof of Lemma \ref{lemmacombi},  define a vocabulary $\tau'$ with two sorts: $s_0$ and $s_1$, with all of the symbols occurring in $\tau$ written in sort  $s_1$. Let $E$ be a binary relation symbol in sort $s_0$, $G$ a function symbol  from $s_1$ to $s_0$, and $c$ a constant symbol  in sort $s_0$.  In addition, for every quantifier-free first-order formula $\psi(x,y_1, \dots, y_n)$ in the language $E,c$, add an $n$-ary function symbol $f_\psi$ of sort $s_0$.
  
  \p Let $\K^{*}$ be the class of all structures 
$\mathcal{M} := \left(M,A,E^\M,G^\M,c^\M, \{f^\mathcal{M}_\psi\} \right)$
such that \begin{enumerate}
\item $(M,E^\mathcal{M})\models \ZFC^{-*}$, % for $n$ big enough so that the argument will go through,
\item $(M,E^\mathcal{M})\in {\Q}_R$, i.e., it is isomorphic to a transitive $R$-correct model, 
\item  $|M| \leq 2^{2^{F(|A|)}}$,
\item $\M \models \Phi(c)$, written  with $E$ instead of $\in$,
\item $\M \models \forall \bar{z}  \: ( \exists x \psi(x, \bar{z}) \to \psi(f_\psi(\bar{z}), \bar{z}) )$ for every quantifier-free $\psi$.
\item $\M \models G$ is a bijection between $A$ and $\{x : x \:E \:c\}$. \end{enumerate}

\p Conditions 1, 4 and 6  are in first-order logic,  whereas 2 is $\DDelta^{\rm B}(\L^*)$-axiomatizable by the equivalent condition (2) of \emph{bounded symbiosis}. Moreover, 3 is $\SSigma^{\mathrm{B}}(\L^{*})$-axiomatisable by the definition of ``definably bounding'' (Definition \ref{definablybounding}), by Example \ref{alephbound}, and the discussion following it.  Finally, while 5 might look like an infinite set of sentences (and we are not assuming that $\L^*$ is infinitary), it is still true that, since $\models$ is $\Delta^\id_1$, the entire condition 5 can be expressed in a $\Delta^\id_1$ way in set theory. By condition (2) of \emph{bounded symbiosis}, the class of models satisfying 5 is $\DDelta^{\rm B}(\L^*)$.  Therefore,  $\K^*$ is $\SSigma^{\rm B}(\L^*)$.

\p Let $\A$ be a structure in $\K$ with $|\A| \geq \kappa$, and   let $\kappa'>  \kappa$ be any cardinal. { Since $\K$ is closed under isomorphisms, we may assume wlog. that $A$ is transitive. } Our goal is to find $\A' \in \K$ such that $|\A'| \geq \kappa'$ and $\A \:\embedFOL \: \A'$.

\p
Let $\theta:= F(\rho_\He(\A))^+ = F(|\A|)^+$, choose Skolem functions $f_\psi^{\He_\theta}: \He_\theta^n \to \He_\theta$, and consider the structure $$\M := (\He_\theta, A ,\in, \mathrm{id}, \A, \{f_\psi^{\He_\psi} \} )$$
 Clearly $\M$ satisfies 1, 5  and 6 of the definition of $\K^{*}$. Moreover, due to Lemma \ref{lemmichka} (1), $\He_\theta$ is $R$-correct and $\Phi(\A)$ is absolute for $\He_\theta$. Hence,  2 and 4 are  satisfied as well. Finally,  3 holds because

$$|\He_\theta| \leq 2^\theta = 2^{F(|\A|)^+} \leq 2^{2^{F(|\A|)}}.$$ 

\p Therefore $\M \in \K^*$.

\p Let $\chi$ be an $\L^*$-sentence in an extended vocabulary $\tau''$ such that  $\K^{*}$ is a ``bounded projection'' of $\Mod(\chi)$. Let $h:\Ord\rightarrow \Ord$ be the function as in Lemma \ref{Lemma:FunctionHBounded}.

%Recall that by boundedness, for every $\mathcal{S} \in \K^{*}_1$ there exists $\lambda_\mathcal{S}$ such that if $\S' \in \K^{*}_1$ and $\S'$ projects to $\S$ then $|\S'| \leq \lambda_\S$ (see Definition \ref{bounded} (2)). Since there are only set-many non-isomorphic models of a fixed cardinality, we can define a function $h: \Ord \to \Ord$ which is monotone and such that, if $\S'\in \K^{*}_1$ and $S'$ projects to $S$, then $|\S'| \leq h(\S)$.

\p Let $\M_1 = (\M, \dots)$ be such that $\M_1 \models \chi$ and  $\M_1 \till \tau' = \M$. Since $|\M_1| \geq |\M| \geq |\A| \geq \kappa$, we can apply $\ULST_\infty(\L^{*})=\kappa$ to find  $\N_1$ such that $\N_1 \models \chi$ and  $|\N_1| > h \left(2^{2^{F(\kappa')}}\right)$, and $\M_1 \subseteq \N_1$ (i.e., $\M_1$ is a substructure of $\N_1$). Let $\N = \N_1 \till \tau'$. %Note that by definition $\N \in \K^{*}$ and $|\N_1| \leq h(|\N|).$ 
We write $\N = (N,B,E^\N, G^\N, c^\N, \{f_\psi^\N\})$ for this structure.

\p Let  $(\bar{N},\in)$ be the transitive collapse of $(N,E^\N)$, and $\overline{c^\N}$ be the image of $c^\N$ under this collapse. We claim that $\A' := \overline{c^\N}$ is the model we are looking for. 

\begin{claim} $\A' \in \K$. \end{claim}

\begin{proof} Since $\N \in \K^*$, we know that $(N,E^\N) \models \Phi(c)$ (written in $E$), and therefore $\bar{N} \models \Phi(\bar{c})$ (written in $\in$). Also, since $(N,E^\N) \in \Q_R$, we know that $\bar{N}$ is $R$-correct, in particular, $\Sigma_1(R)$ formulas are upwards absolute. Therefore $\Phi(\A')$ is true, so $\A' \in \K$. 
\end{proof}

\begin{claim}
$\kappa' < |\A'|$.
\end{claim}
\begin{proof} By the definition of the function $h$ as in Lemma \ref{Lemma:FunctionHBounded}, we know that $|\N_1| \leq h(|\N|)$. Thus we have

$$h \left(2^{2^{F(\kappa')}}\right) < |\N_1| \leq h(|\N|)$$ and since $h$ is order-preserving,  $2^{2^{F(\kappa')}} < |\N|$. By condition 3 of the definition of $\K^*$, we have $|N| \leq 2^{2^{F(|B|)}}$. Therefore $2^{2^{F(\kappa')}} < |N| \leq 2^{2^{F(|B|)}}$, from which it follows that $\kappa' < |B|$. Finally, by condition 6 we get that $|B| = |\{x \in N : x E^\N c^\N| = |\A'|$. \end{proof}

\begin{claim}
There is an $\L_{\omega\omega}$-elementary embedding from $\A$ to $\A'$.
\end{claim}
\begin{proof} By condition 5, both models $\M_1$ and $\N_1$ satisfy the axioms for Skolem functions concerning first-order quantifier-free formulas in $\{E,c\}$. In addition, since $\M_1$ is a \emph{substructure} of $\N_1$, the interpretations of  $f_\psi$ coincide between the models, i.e., $f_\psi^{\N_1} \till \He_\theta= f_\psi^{\He_\theta} $ for every $\psi$. Thus, if $\N_1 \models \exists x \psi(x,\bar{z})$ and $\bar{z} \in \He_\theta$, then $\N_1 \models \psi(f_\psi(\bar{z}),\bar{z})$, so $(\He_\theta, \in, \A) \models  \psi(f_\psi(\bar{z}),\bar{z})$. It follows that $\N_1$ and $(\He_\theta,\in,\A)$ satisfy the same $\Sigma_1$-formulas in $\{E,c\}$.

\p Let $\pi:N\rightarrow\bar{N}$ be the collapsing map. Since the first-order satisfaction relation is $\Delta_1$, for every first-order $\phi$ and for every  $\bar{a} = a_1,\ldots a_n \in \A$  we have  

\p $\A \models \phi(\bar{a})$

\p $\;\;\Leftrightarrow \;\; \He_{\theta} \models (\A \models \phi(\bar{a}))$

\p  $\;\;\Leftrightarrow \;\; \N_1 \models (c \models \phi(\bar{a}))$

%\p $\;\;\Leftrightarrow \;\;  \N\models (c \models \phi(\bar{a}))$ 

\p $\;\;\Leftrightarrow \;\;  (\bar{N},\in, \A') \models ( c \models \phi(\pi(\bar{a})))$ 

\p $\;\;\Leftrightarrow \;\;  \A' \models \phi(\pi(\bar{a})))$. 

\p Hence $\pi {\upharpoonright A}:\A\preccurlyeq_{\L_{\omega\omega}}\A'$ as required.\footnote{In this proof we have occasionally identified the syntax and semantics of first-order logic for ease of readability.}
\end{proof}
}

\end{proof}

%{\color{blue} Remark: }

\section{The predicate $\PwSt$ and second order logic} \label{SectionLargeCardinals}

In this section we apply our results to determine upper and lower bounds for the large cardinal strength of $\USR(\PwSt)$ and $\ULST_\infty(\L^2)$, which will also yield upper bounds for    other symbiotic pairs $\L^*$ and $R$. The main point is that  $\PwSt$ can be seen as an upper bound for all $\Pi_1$ predicates. The following is not hard to verify (see \cite[Section 6.5]{GaleottiThesis} for the details).

\begin{Fact}\label{CorBoundedH}
The function $\alpha\mapsto \Ve_\alpha$ is $\Sigma^{F}_1(\PwSt)$-definable (for suitable $F$).  Also, the function $\mathrm{\He}$ that maps every infinite successor cardinal $\theta$ to $\He_{\theta}$ is $\Sigma^{F}_1(\PwSt)$-definable. 
\end{Fact} 

\begin{lemma}\label{Lemm:PIbound} For every $\Pi_1$ predicate  $R$, if $\phi$ is   $\Sigma^F_1(R)$ then it is $\Sigma^F_1(\PwSt)$.
\end{lemma}

\begin{proof} Suppose $\phi$ is $\Sigma_1^F(R)$. Then for every $a$ we have 

$$\phi(a) \; \Leftrightarrow \; \exists \He_\theta \: \left(\rho_\He(\He_\theta) < 2^{2^{F(\rho_\He(a))}} \; \: \land \; \: \He_\theta \models \phi(a)\right).$$
By the previous fact, ``being $\He_\theta$'' is $\Sigma^{F'}_1(\PwSt)$-definable (possibly another $F'$). In conjunction with the upper bound, the whole expression is also $\Sigma^{F''}_1(\PwSt)$-definable (for   $F''$ being  the maximum of  $F'$ and $\alpha \mapsto 2^{2^{F(\alpha)}}$). To see that the equivalence holds, let $\theta = F(\rho_\He(a))^+$. Then $\rho_\He(\He_\theta) \leq 2^\theta  \leq 2^{2^{F(\rho_\He(a))}}$, and moreover $\phi(a)$ is absolute for $\He_\theta$ by Lemma \ref{lemmichka} (2). \end{proof}

\begin{corollary} \label{corol} $\;$

 \begin{enumerate}

\item For every $\Pi_1$ predicate $R$ we have $\USR(R) \leq \USR(\PwSt)$. In particular, if $\USR(\PwSt)$ is defined then $\USR(R)$ is also defined. 

\item If  $\L^*$ is any logic which is \emph{ boundedly symbiotic} to some $\Pi_1$-predicate $R$, has $\Delta_0$-definable syntax and $\dep(\L^*) = \omega$, then $\ULST_\infty(\L^*) \leq \ULST_\infty(\L^2)$. In particular, if $\ULST_\infty(\L^2)$ is defined, then so is $\ULST_\infty(\L^*)$. \end{enumerate} \end{corollary}

A famous result of Magidor \cite{Magidor} shows that the least cardinal $\kappa$ for which $\L^2$ satisfies a $\kappa$-version of compactness, is the   least extendible cardinal. One can show that this version of compactness implies $\ULST_\infty(\L^2) = \kappa$. Therefore an extendible cardinal provides an upper bound for $\ULST_\infty(\L^2)$ and  $\USR(\PwSt)$, as well as other  pairs $\L^*$ and $R$ satisfying bounded symbiosis and the conditions of Theorem \ref{SymbiosisTheorem}. For completeness, we include a short proof of this fact. % using the principle $\USR(\PwSt)$ Theorem \ref{SymbiosisTheorem}. %
%I%n fact, \emph{symbiosis} allows us to provide a more direct proof of this fact. Moreover, by Corollary \ref{corol}, an extendible is also an upper bound for all other logics and $\Pi_1$ %predicates which satisfy bounded symbiosis and the conditions of Theorem \ref{SymbiosisTheorem}. 

\begin{theorem}[Magidor \cite{Magidor}]  \label{Thm:PWUpperBOU}
If $\kappa$ is an extendible cardinal, then $$\USR(\PwSt) = \ULST(\L^2) \leq \kappa.$$ Moreover, $\USR(R) \leq \kappa$ for  every $\Pi_1$ predicate $R$, and $\ULST_\infty(\L^*) \leq \kappa$ for any $\L^*$ which is boundedly symbiotic with some $\Pi_1$ predicate, and which has $\Delta_0$-definable syntax and $\dep(\L^*) = \omega$. 
\end{theorem}
\begin{proof} Let $\kappa$ be extendible, and we will prove that $\USR(\PwSt)  \leq \kappa$. The other statements follow by Theorem \ref{SymbiosisTheorem} and Corollary \ref{corol}.

\p Let $\K$ be a $\Sigma_1^{F}(\PwSt)$-definable class of $\tau$-structures closed under isomorphisms. Fix some $\A \in \K$ with $|\A|\geq \kappa$. Let $\kappa'> \kappa$ be arbitrary. Let $\eta>\kappa'$ be such that $\A\in \Ve_\eta$ and $\Ve_\eta\models \Phi(\A)\land (|\A|\geq\kappa)$. Then there is an elementary embedding $J: \Ve_\eta \embed \Ve_\theta$ for some  $\theta$, and $J(\kappa)>\eta>\kappa'$. But then by elementarity we have  $\Ve_\theta\models \Phi(J(\A)) \land |J(\A)|\geq J(\kappa)$. Since $\Ve_\theta$ is $\PwSt$-correct, $\Phi(J(\A))$ holds, so $J(\A) \in \K$. Moreover, since $\theta$ is sufficiently large, we have $|\A| \geq J(\kappa) > \eta > \kappa'$. Finally,  $\A \; \embedFOL \; J(\A)$ holds by elementarity and first-order definability of ``$\A \models \phi$''.  \end{proof}

%\begin{corollary}\label{CorL2}
%If $\kappa$ is an extendible cardinal, then  $\ULST_\infty(\L^2) \leq \kappa$. Moreover $\ULST_\infty(\L^*) \leq \kappa$ for any $\L^*$ which is boundedly symbiotic with some $\Pi_1$ %predicate, and which has $\Delta_0$-definable syntax and $\dep(\L^*) = \omega$.

%\end{corollary}

Now we look at how much large cardinal strength we can obtain from $\USR(\PwSt)$. %, which, by Theorem \ref{SymbiosisTheorem}, also yields a bound for $\ULST_\infty(\L^2)$. {\color{blue} CONJECTURE: we get an extendible cardinal... is anything known about this??}

\begin{theorem}\label{Theo:ExtN}
If $\USR(\PwSt)$ is defined, then there exists an $n$-extendible cardinal for every natural number $n>0$.
\end{theorem}
\begin{proof}

%a $(M,a,E),$
%where
%\begin{enumerate}
%\item $a$ is an ordinal,
%\item $M=\Ve_{a+n}$,
%\item $E={\in}{\till} \Ve_{a+n}$.
%\end{enumerate}

Assume that $\USR(\PwSt)=\kappa$. Let  $\K$ to be the class of all structures  which are   isomorphic to $(\Ve_{\alpha+n}, \in, \alpha)$, in the language $\{E,a\}$. 

\p By Fact \ref{CorBoundedH}, being a structure of the form $(\Ve_{\alpha+n}, \in, \alpha)$ is $\Sigma_1^{F}(\PwSt)$-definable. Then $(M,a,E) \in \K$ iff $\exists (\Ve_{\alpha+n}, \in, \alpha)$ and $\exists f: (M,a,E) \cong  (\Ve_{\alpha+n}, \in, \alpha)$. Notice that also $\rho_\He(\Ve_{\alpha+n}) \leq \rho_\He(M)$. Thus, $\K$ is $\Sigma_1^F(\PwSt)$-definable.

\p Take any $\mu \geq \kappa$. Since $( \Ve_{\mu+n},\mu, {\in}) \in \K$,   by $\USR(\PwSt)$ there exists an elementary embedding  
$$J:( \Ve_{\mu+n},\mu, \in)\embed_{\L_{\omega\omega}} ( \Ve_{\beta+n},\beta,\in)$$ for some $\beta>\mu$, which maps $\mu$ to $\beta$. Let $\lambda$ be the critical point of $J$, which is $\leq \mu$. But then $J{\upharpoonright} \Ve_{\lambda+n}:\Ve_{\lambda+n}\embed \Ve_{J(\lambda)+n}$ (this includes the case $\lambda=\mu$), since:% Indeed, for every formula $\varphi$ and $x_1,\ldots,x_n\in \Ve_{\lambda+n}$ we have 
\begin{align*}
\Ve_{\lambda+n}\models \varphi(x_1,\ldots,x_n)  &\Leftrightarrow \Ve_{\mu+n}\models ( \Ve_{\lambda+n} \models \varphi(x_1,\ldots,x_n)) \\
&\Leftrightarrow \Ve_{\beta+n}\models ( J(\Ve_{\lambda+n}) \models \varphi(J(x_1),\ldots,J(x_n)))  \\
&\Leftrightarrow\Ve_{\beta+n}\models  ( \Ve_{J(\lambda)+n} \models \varphi(J(x_1),\ldots,J(x_n))) \\ 
&\Leftrightarrow \Ve_{J(\lambda)+n}\models \varphi(J(x_1),\ldots,J(x_n))
\end{align*}  

\p Since $n< J(\lambda)$, it follows that $\lambda$ is $n$-extendible. \qedhere

\end{proof} 

%As a direct consequence of Theorem \ref{SymbiosisTheorem} and bounded symbiosis between $\PwSt$ and $\L^2$, we immediately obtain:

\begin{corollary}\label{CorL2L}
If $\ULST_\infty(\L^2)$ is defined then there exists an $n$-extendible cardinal for every  $n$.
\end{corollary}

Notice that the only reason that the proof works for  $n < \omega$ and not arbitrary $\alpha$, is because the class $\K$ needs to be definable.  It is easy to adapt the proof to show that there exists a $\gamma$-extendible cardinal for any $\Sigma_1^{F}(\PwSt)$-definable ordinal  $\gamma$. In fact, we conjecture that the consistency strength is exactly an extendible.

\begin{Conjecture} \label{conji} $\USR(\PwSt)$ and $\ULST_\infty(\L^2)$ are defined if and only if there exists and extendible cardinal. 
\end{Conjecture}

  \label{DiscussionBuond}  %Note that these bounds do not coincide; see, e.g., \cite[p.317]{kanamori}.

\section{Questions and concluding remarks} \label{SectionQuestions}

The biggest question left open in this paper is the exact consistency strength of $\USR(\PwSt)$ and $\ULST_\infty(\L^2)$, i.e., Conjecture \ref{conji}.

\bigskip Other questions that we have not investigated involve a similar analysis of the large cardinal strength for other symbiotic pairs $\L^*$ and $R$.

\begin{Question} What is the large cardinal strength (or, at least, lower and upper bounds), for the principles $\USR(R)$ and $\ULST_\infty(\L^*)$, for other boundedly symbiotic pairs $R$ and $\L^*$, such as the ones in \cite[Proposition 4]{Symbiosis}? \end{Question}

 Another important issue, which we have not investigated in this paper, is the study of various \emph{compactness} properties of strong logics.

\begin{definition}
A logic $\L^{*}$ is $(\kappa,\gamma)$-compact if for every set $T$ of sentences of size $\gamma$, if every $<\kappa$-sized subset of $T$ has a model, then $T$ has a model. A logic $\L^*$ is $(\kappa,\infty)$-compact if it is $(\kappa,\gamma)$-compact for every $\gamma$. Classical compactness is $(\omega,\infty)$-compactness.
\end{definition}

Most strong logics are not $(\omega,\infty)$-compact but may be $(\kappa,\infty)$-compact for some $\kappa$. Often such a  $\kappa$ will exhibit large cardinal properties, e.g., Magidor's result on $\L^2$ \cite{Magidor}. As we mentioned in the previous section, it is easy to see that:

\begin{center} If  $\L^{*}$ is $(\kappa,\infty)$-compact then $\ULST_{\infty}(\L^{*})\leq \kappa$. \end{center}

We do not know whether the converse holds. The following questions seem interesting and worth investigating:

\begin{Question}  Assume that $\kappa$ is a regular cardinal. For which logics does $\ULST_{\infty}(\L^{*})\leq \kappa$ imply  $(\kappa,\infty)$-compactness? \end{Question}

One can try to look for a set-theoretic principle involving $\Delta_1(R)$ definable classes of structures, which would correspond to $(\kappa,\infty)$-compactness in a similar way as in Theorem \ref{SymbiosisTheorem}.

\begin{Question} Is there a set-theoretic principle $P(R)$, for classes definable using a $\Pi_1$-parameter $R$, such that if   $R$ and $\L^{*}$ are (bounded) symbiotic, then $P(R) = \kappa$ if and  only if $\L^{*}$ is $(\kappa,\infty)$-compact? \end{Question}

Answering the last question could involve extensions of partial orders within a fixed $\Delta_1(R)$-class, using  ideas from \cite{MakowskyShelah}.  Notice, however, that when dealing with compactness, large vocabularies are essential, so the corresponding principles will require the use of parameters, which will restrict the class of logics $\L^*$ to which it can apply.

\bigskip 

\noindent{\bf Acknowledgements:} We would like to thank Soroush Rafiee Rad and Robert Passmann for initiating this research and providing valuable input. 

\bibliographystyle{plain}
\bibliography{REFERENCES}

\begin{thebibliography}{10}

\bibitem{BagariaCn}
Joan Bagaria.
\newblock {$C^{(n)}$}-cardinals.
\newblock {\em Arch. Math. Logic}, 51(3-4):213--240, 2012.

\bibitem{Symbiosis}
Joan Bagaria and Jouko V\"{a}\"{a}n\"{a}nen.
\newblock On the symbiosis between model-theoretic and set-theoretic properties
  of large cardinals.
\newblock {\em J. Symb. Log.}, 81(2):584--604, 2016.

\bibitem{BarwiseBook}
J.~Barwise and S.~Feferman, editors.
\newblock {\em Model-theoretic logics}.
\newblock Perspectives in Logic. Association for Symbolic Logic, Ithaca, NY;
  Cambridge University Press, Cambridge, 2016.
\newblock For the first (1988) edition see [ MR0819531].

\bibitem{BarwiseAdmissibleSets}
Jon Barwise.
\newblock {\em Admissible Sets and Structures}.
\newblock Perspectives in Logic. Cambridge University Press, 2017.

\bibitem{GaleottiThesis}
Lorenzo Galeotti.
\newblock {\em The theory of generalised real numbers and other topics in
  logic}.
\newblock PhD thesis, Hamburg University, 2019.
\newblock ILLC Dissertation Series DS-2019-04.

\bibitem{Levy}
Azriel L\'{e}vy.
\newblock A hierarchy of formulas in set theory.
\newblock {\em Mem. Amer. Math. Soc.}, 57:76, 1965.

\bibitem{Lindstrom1966}
Per Lindstr\"{o}m.
\newblock First order predicate logic with generalized quantifiers.
\newblock {\em Theoria}, 32:186--195, 1966.

\bibitem{Magidor}
Menachem Magidor.
\newblock On the role of supercompact and extendible cardinals in logic.
\newblock {\em Israel J. Math.}, 10:147--157, 1971.

\bibitem{MagidorJouko}
Menachem Magidor and Jouko V\"{a}\"{a}n\"{a}nen.
\newblock On {L}\"{o}wenheim-{S}kolem-{T}arski numbers for extensions of first
  order logic.
\newblock {\em J. Math. Log.}, 11(1):87--113, 2011.

\bibitem{MakowskyShelah}
Johann~A. Makowsky and Saharon Shelah.
\newblock Positive results in abstract model theory: a theory of compact
  logics.
\newblock {\em Annals of Pure and Applied Logic}, 25(3):263 -- 299, 1983.

\bibitem{Makowsky1976}
Johann~A. Makowsky, Saharon Shelah, and Jonathan Stavi.
\newblock {$\Delta$}-logics and generalized quantifiers.
\newblock {\em Ann. Math. Logic}, 10(2):155--192, 1976.

\bibitem{JoukoThesis}
Jouko V\"a\"an\"anen.
\newblock {\em Applications of set theory to generalized quantifiers}.
\newblock PhD thesis, University of Manchester, 1967.

\bibitem{Jouko1979}
Jouko V\"{a}\"{a}n\"{a}nen.
\newblock Abstract logic and set theory. {I}. {D}efinability.
\newblock In {\em Logic {C}olloquium '78 ({M}ons, 1978)}, volume~97 of {\em
  Stud. Logic Foundations Math.}, pages 391--421. North-Holland, Amsterdam-New
  York, 1979.

\bibitem{VaananenBounded}
Jouko V\"{a}\"{a}n\"{a}nen.
\newblock {$\Delta $}-extension and {H}anf-numbers.
\newblock {\em Fund. Math.}, 115(1):43--55, 1983.

\end{thebibliography}

\end{document}